\documentclass[12pt]{amsart}
\usepackage{a4wide,pdfsync,hyperref}
\usepackage[utf8]{inputenc}
\usepackage{amsmath,amssymb,esint}
\usepackage{enumerate}
\usepackage{color}

\newtheorem{theorem}{Theorem}[section]
\newtheorem{lemma}[theorem]{Lemma}

\newtheorem{proposition}[theorem]{Proposition}

\newcommand\RR{{{\mathbb R}}}

\newcommand{\reff}[1]{(\ref{#1})}

\newcommand{\p}{\partial}
\newcommand{\La}{\Delta}

\newcommand{\beq}{\begin{eqnarray*}}
\newcommand{\enq}{\end{eqnarray*}}
\newcommand{\beqq}{\begin{eqnarray}}
\newcommand{\enqq}{\end{eqnarray}}
\newcommand{\ben}{\begin{equation}\label}
\newcommand{\enn}{\end{equation}}
\newcommand{\bef}{\begin{proof}}
\newcommand{\enf}{\end{proof}}

\title{Effective viscosity of a polydispersed suspension}
\author{Matthieu Hillairet}
\address{Institut Montpelli\'erain Alexander Grothendieck, CNRS, Univ. Montpellier, France}
\email{matthieu.hillairet@umontpellier.fr}
\author{Di Wu}
\address{School of Mathematical Sciences, Peking University, Beijing 100871, P.R. China.}
\email{wudi@math.pku.edu.cn}
\date{\today}

\begin{document}

\maketitle

\begin{abstract}
We compute the first order correction of the effective viscosity for a suspension containing solid particles with arbitrary shapes. We rewrite the computation as an homogenization problem for the Stokes equations in a perforated domain.
Then, we extend the method of reflections \cite{Velazquez,Niet-Schub} to approximate the solution to the Stokes problem with a fixed number of particles. By obtaining sharp estimates, we are able to prove that this method converges for small volume fraction of the solid phase whatever the number of particles. This allows to address the limit when the number of particles diverges while their radius tends to $0.$ We obtain a system of PDEs similar to the Stokes system with a supplementary term in the viscosity proportional to the volume fraction of the solid phase in the mixture. 
\end{abstract}

\section{Introduction}

When a viscous fluid transports solid particles, the particles modify in return the properties of the fluid. For instance, the rheological properties of the fluid are altered.
In his seminal paper \cite{Einstein}, Einstein  addresses the computation of the effective viscosity of the mixture, having in mind it could help recovering the size  of the transported particles. He obtains then the formula: 
\begin{equation} \label{eq_einstein}
\mu_{eff} = \mu \left(1 + \dfrac{5}{2}\phi + o(\phi)\right).
\end{equation}
Here $\mu$ stands for the (bulk) viscosity of the incompressible fluid alone, 
$\mu_{eff}$ denotes the viscosity of the mixture and $\phi$ stands for the 
volume fraction of the solid suspension of spheres. Einstein formula has been the subject of numerous studies:  analysis of Einstein "formal" computations \cite{Almog-Brenner,Keller-Rubenfeld,Ammari,Haines-Mazzucato}, computation of second order expansion \cite{Hinch,Batchelor-Green,GVH}. We refer the reader also to 
\cite{JeffreyAcrivos} for a comprehensive picture on the possible phenomena influencing the effective viscosity of a suspension. Most of these studies  consider homogeneous suspensions.  However, as mentioned in \cite{JeffreyAcrivos}, a formula for the effective viscosity depending only on the volume fraction is hopeless to describe general suspensions, the factor $5/2$ in the above formula being in particular valid for a suspension of spheres {\em a priori}. In this paper, we provide a method for the computation of an effective viscosity allowing a distribution of shapes for the particles in the suspension.

\medskip

A second motivation of the paper is to obtain a "local" formula for the effective viscosity similar to \cite{Almog-Brenner,Niet-Schub}. To be more precise, we rephrase now the computation of an effective viscosity, as depicted in \cite{Batchelor},  into a homogenization problem. We consider an incompressible newtonian fluid occupying the whole space $\mathbb R^3$ and transporting a cloud made of $N$ particles. We neglect the particle and 
fluid inertia so that computing an effective viscosity amounts to understand 
the behavior of the system when it is submitted to a strain flow $x \mapsto Ax$
(where $A$ is a symmetric trace-free matrix). This reduces to the following
{\em stationary}  problem. We  denote by $(u,p)$ the fluid velocity-field/pressure.
The domain of the $l$-th solid particle is the smooth bounded open set $B_l \subset \mathbb R^3$ and its center of mass is $x_l.$ The motion of $B_l$ is associated to a pair of translational/rotational velocities $(V_l,\omega_l)$. Introducing $\mu$ the viscosity of the fluid, the unknowns $(u,p,(V_l,\omega_l)_{l=1,\ldots,N})$ are computed by solving the problem
\begin{align}
\label{eq_Stokes}
& \left\{
\begin{aligned}
- {\rm div}\, \Sigma_{\mu}(u,p) &= 0 \\
{\rm div}\, u &= 0 
\end{aligned}
\right.
\qquad \qquad
\text{ in $\mathbb R^3 \setminus \bigcup_{l=1}^N \overline{B}_l,$}
\\
\label{bc_Stokes}
& \left\{
\begin{aligned}
u(x) &= V_l + \omega_l \times (x-x_l) && \text{ on $\partial B_l,$ for $l=1,\ldots,N$} \\ 
 u(x) &= Ax && \text{ at infinity,}
\end{aligned}
\right.
\\
\label{eq_Newton}
& \left\{
\begin{aligned}
\int_{\partial B_l} \Sigma_{\mu}(u,p) n {\rm d}s &= 0 \\
\int_{\partial B_l} (x-x_l) \times \Sigma_{\mu}(u,p) n {\rm d}s &= 0 \\
\end{aligned}
\right.
\qquad \qquad  \text{for }\, l = 1,\ldots,N.
\end{align}
In this system, we introduced the fluid stress-tensor $\Sigma_{\mu}(u,p).$
Under the assumption that the fluid is newtonian, it reads:
\[
\Sigma_{\mu}(u,p) = 2\mu D(u) - p \mathbb I_{3} = \mu (\nabla u + \nabla^{\top} u) - p \mathbb I_3.
\]
The zero source terms on the right-hand side of the first equation in \eqref{eq_Stokes} and both equations of  \eqref{eq_Newton} are reminiscent of the intertialess assumption. The second equation of \eqref{bc_Stokes} must be understood as
\[
\lim_{x \to \infty} |u(x) - Ax| = 0.
\]
In the last equations \eqref{eq_Newton} the symbol $n$ stands for the normal to 
$\partial B_l.$ By convention, we assume that it points inwards the solid $B_l$
and outwards the fluid domain that we denote $\mathcal F_N$ in what follows:
\[
\mathcal F_N =  \mathbb R^3 \setminus \bigcup_{l=1}^N \overline{B}_l.
\]

Under the assumption that the $B_l$ do not overlap, existence/uniqueness of a solution to \eqref{eq_Stokes}-\eqref{bc_Stokes}-\eqref{eq_Newton} falls into the scope of the classical theory for the Stokes equations (see \cite[Section V]{Galdi}).  We give a little more details in the  next section. We only mention here that the pressure is unique up to a constant. But, this has no impact on our computations and we consider the pressure as being uniquely defined below (this problem could be fixed by assuming that one mean of the pressure has  a fixed value). Our aim is to tackle the asymptotics of this solution when the $B_l$ are small and many. To make this statement quantitative, we introduce further assumptions regarding the $B_l$. 
Namely, we assume that there exists a diameter $a>0,$ centers $x_l \in \mathbb R^3$ and shapes $\mathcal B_l$ (meaning smooth bounded connected open sets of $\mathbb R^3$) such that
\begin{equation} \tag{H1} \label{H1}
\mathcal B_l \subset B(0,1), \qquad \int_{\mathcal B_l} x{\rm d}x = 0, \qquad 
B_l = x_l + a \mathcal B_l, \qquad \forall \, l=1,\ldots,N.
\end{equation}
Then, we prescribe that the solid domains remain in a compact set $K$ and that there is no-overlap between the particles: 
\begin{equation} \tag{H2} \label{H2}
B_l \subset K \quad \forall \, l =1,\ldots,N \,, \qquad d := \min_{l\neq \lambda} |x_l-x_\lambda| > 4 a.
\end{equation}
With these conventions, we note that the total volume of the solid phase
is at most $4Na^3\pi/3$ so that, globally, in the volume $K,$ the volume fraction of the solid phase is controlled by $4Na^3\pi/3|K|.$ However, the separation assumption \eqref{H2} implies  that we have also  a uniform local control of the solid phase volume fraction by $a^3/d^3.$ We use also
constantly below that, with \eqref{H1}-\eqref{H2}, we obtain $N \leq C|K|/ d^{3}.$

\medskip

In order to derive an effective viscosity for the mixture, the classical point of view proposed in \cite{Einstein,Batchelor} is to compute the rate of work of the viscous stress tensor on the boundary $\partial K$ of the domain $K$ containing the solid particles: 
\[
W_{eff} := \int_{\partial K} \Sigma_{\mu}(u,p) n \cdot Ax {\rm d}s 
\]
and to compare the excess with respect to the value $W_0 = 2 \mu A:A |K|$
 that would yield in case there is no particle. In brief, the analysis of Einstein -- in the case
 the $B_l$ are spheres of radius $a$ filling a bounded domain -- relies on splitting the solution $(u,p)$ into $u = u_0+u_1,$ $p=p_0+p_1.$ Here $(u_0,p_0)$ is the pure strain applied on the boundaries at infinity:
 \[
 u_0(x) = Ax \qquad p_0(x) = 0,
 \] 
 (this is a solution to the Stokes equations on $\mathbb R^3$ since $A$ is trace-free) while the term $(u_1,p_1)$ compensates the trace of the boundary conditions $u-u_0(x) = -Ax$ on the $B_l$ that cannot be
matched by a suitable pair $(V_l,\omega_l)$ in \eqref{bc_Stokes}. Namely, one may write:
\[
u(x)-u_0(x) = Ax_l - A(x-x_l) \qquad \text{ on $\partial B_l$}.
\]  
Since $A$ is symmetric the latter linear term in the boundary condition cannot be compensated by a rigid rotation. Under the assumption that the holes are well-separated one provides the approximation:
\[
u_1(x) = \sum_{l=1}^{N} U^a[A](x-x_l), \qquad p_1(x) = \sum_{l=1}^N P^a[A](x-x_l),
\]
where $(U^a[A],P^a[A])$ is the solution to the Stokes problem \eqref{eq_Stokes} outside $B(0,a)$ with boundary
condition $U^a(x) = -Ax$ on $\partial B(0,a)$ (and vanishing boundary conditions at infinity). With this formula at-hand, one obtains that 
\[
W_{eff}  = W_0 + \sum_{l=1}^N \int_{\partial K} \Sigma_{\mu}(U^a[A](\cdot -x_l),P^a[A](\cdot-x_l))n \cdot Ax{\rm d}s.
\]
Via conservation arguments related to the divergence form of the Stokes equation, the boundary
integrals involved in $W_{eff}$ can be transformed into $N$ integrals over the boundaries of $B(x_l,a).$
It is then possible to apply the explicit value of the solution $(U^a,P^a).$
Summing the contributions of all the particles leads finally to the first order expansion:
\[
W_{eff} =  2 \mu A:A \left( |K| + \dfrac{1}{2}  \dfrac{20\pi Na^3}{3} \right), 
\]
leading to formula \eqref{eq_einstein}.
We refer the reader to \cite[p.246]{Batchelor}  for more details on this computation.
 
\medskip

 Herein, we show that solutions to \eqref{eq_Stokes}-\eqref{bc_Stokes}-\eqref{eq_Newton}  are close to solutions to the continuous analogue:
 \begin{align} \label{eq_continu}
  & \left\{  
 \begin{aligned}
 - {\rm div} (2\mu [(1+ \mathbb M_{eff})(D(u))]  - p \mathbb{I}_3) &=& 0 \\
 {\rm div} u &=& 0 
 \end{aligned}
 \right. && \text{ on $\mathbb R^3$}, \\ \label{bc_continu}
  & \quad u(x) = A x && \text{ at infinity}. 
 \end{align}
Here the symbol $(1+\mathbb M_{eff})(D(u))$ stands for $D(u) + \mathbb {M}_{eff}(D(u)),$ with $\mathbb{M}_{eff}$ an application that maps $D(u)$  to the $3\times 3$  matrix $\mathbb{M}_{eff}(D(u))$. This linear mapping measures the collective reaction of the particles to the strain induced by $D(u).$
We emphasize that we allow this mapping to depend on the space variable $x.$ 
To be more precise, we explain now the computation of $\mathbb M_{eff}.$ 
For arbitrary $l \in \{1,\ldots,N\}$ let denote by $(U[A,B_l],P[A,B_l])$ the unique solution 
to 
\begin{align} \label{eq_Stokeslocal}
& \left\{
\begin{aligned}
- {\rm div \Sigma(u,p)}&=& 0  && \text{ in $\mathbb R^3 \setminus \overline{B_l}$},\\
{\rm div} u &=& 0 &&\text{ in $\mathbb R^3 \setminus \overline{B_l}$},
\end{aligned}
\right.
&& 
\left\{
\begin{aligned}
& u(x) = -Ax + V + \omega \times x && \text{ on $\partial B_l$}, \\
& u(x) = 0 && \text{ at infinity },
\end{aligned}
\right. 
\\ \label{eq_Newtonlocal}
& \int_{\partial B_l} \Sigma(u,p)n{\rm d}\sigma = 0 &&
 \int_{\partial B_l} (x-x_l) \times \Sigma(u,p)n{\rm d}\sigma = 0. 
\end{align}
In this system, we have $\mu=1$ but we drop the index $1$ of the symbol $\Sigma$ for legibility.
We note that $(V,\omega)$ are also unknowns in this problem. But they are the lagrange multipliers of the constraint \eqref{eq_Newtonlocal}, so that we may retain only $(U[A,B_l],P[A,B_l])$ as the solution. We associate to this solution:
\[
\mathbb M[A,B_l] := \mathbb P_{3,\sigma} \left[  \int_{\partial \mathcal B_l} - \Sigma(U[A,B_l],P[A,B_l])n \otimes (x-x_l) + 2 U[A,B_l] \otimes n {\rm d}s\right],
\]
where $ \mathbb P_{3,\sigma}$ stands for the orthogonal projection (w.r.t. matrix contraction) on the space of symmetric trace-free $3\times 3$ matrices
${\rm Sym}_{3,\sigma}(\mathbb R)$.  As shown in Section \ref{sec_cellpbm} below the matrix $\mathbb M[A,B_l]$ encodes the far-field decay of the solution $U[A,B_l]$ in the sense that:
\begin{equation} \label{eq_farfield}
U_i[A,B_l](x) =  \mathbb M[A, B_l] : \nabla \mathcal U^i (x) + l.o.t \quad \text{ for i=1,2,3} 
\end{equation}
at infinity (where $\mathcal U^i$ contains vector-fields build up from the Green-function for the Stokes problem).
Due to the linearity of the Stokes equations, we have that, for fixed $B_l$
the mapping $A \mapsto \mathbb M[A,B_l]$ is linear and thus given by a mapping $\mathbb M[B_l]: {\rm Sym}_{3,\sigma}(\mathbb R) \to {\rm Sym}_{3,\sigma}(\mathbb R)$ (such a mapping can be identified with a $5\times 5$ matrix). We set then:
\[
\mathbb M_N(x) = \dfrac{3}{4\pi a^3}\sum_{l=1}^N\mathbb M[{B}_l] \mathbf{1}_{B(x_l,a)} (x)= \dfrac{3}{4\pi}\sum_{l=1}^N\mathbb M[\mathcal{B}_l] \mathbf{1}_{B(x_l,a)}(x) 
\quad \forall \, x \in \mathbb R^3.
\]
We shall obtain below -- under assumption \eqref{H1}-\eqref{H2}-- that $|\mathbb M[\mathcal B_l] |\leq C$ independent of the shape $\mathcal B_l.$ 
The mapping-function $\mathbb M_N$ has then support in $K$ with $\|\mathbb M_N\|_{L^{1}(\mathbb R^3)} \lesssim a^3/d^3$ so that it is bounded independent of $N.$ Then, one can think $\mathbb M_{eff}$ as a possible weak limit if the parameter $N$ was tending to $\infty$. 

\medskip

For instance, in the case $B_l$ are spheres of radius $a$ (so that 
$\mathcal B_l$ is a sphere of radius $1$) comparing the 
expansion \eqref{eq_farfield} with the explicit solutions to the Stokes problem
(see \cite[p. 39] {Guazzelli}) we obtain that $\mathbb M[A,\mathcal B_l] = 20\pi A/3$
so that
\[
\mathbb M_N \sim 5  \sum_{l=1}^N \mathbf{1}_{B(x_l,a)}. 
\]
In this case, the convergence of $\mathbb M_N$ reduces to the convergence of  the distribution of centers $(x_l)_{l=1,\dots,N}.$ If the empirical measures
associated to the distribution of centers converges to some $f \in L^1(\mathbb R^3),$ we obtain, with $\phi= 4\pi Na^3/(3|K|)$ the volume fraction of particles :
 \begin{equation} \label{eq_casesphere}
\mathbb M_N \rightharpoonup  5 \phi f \text{ in $L^1(\mathbb R^3)-w.$}
\end{equation}

\medskip

We give herein a quantitative result with explicit stability bounds for the distance between solutions  to the perforated problem \eqref{eq_Stokes}-\eqref{bc_Stokes}-\eqref{eq_Newton} and to the continuous problem
\eqref{eq_continu}-\eqref{bc_continu}.   We restrict below to  functions $\mathbb M_{eff}$ in classes
\[
\mathcal M(\varepsilon) := \left\{ \mathbb M \in L^{\infty}(\mathbb R^3;{\rm Mat}_{5}(\mathbb R)), \text{ s.t. }  Supp(\mathbb M) \subset K \text{ and } \|\mathbb M\|_{L^{\infty}(K)} \leq \varepsilon \right\}.
\]
Here $\varepsilon >0$ is a given parameter related to the volume fraction $a^3/d^3.$ We identify the space of linear mappings ${\rm Sym}_{3,\sigma}(\mathbb R) \to {\rm Sym}_{3,\sigma}(\mathbb R)$ with ${\rm Mat}_{5}(\mathbb R).$ With the notations introduced before, a precise statement of our main result is the following theorem

\begin{theorem} \label{thm_main}
Let  \eqref{H1}-\eqref{H2} be in force and denote by $(u_N,p_N)$ the unique solution to 
\eqref{eq_Stokes}-\eqref{bc_Stokes}-\eqref{eq_Newton}. Let $\varepsilon_0 >0,$ $\mathbb M_{eff} \in \mathcal M(\varepsilon_0)$ and denote by $(u_c,p_c)$ the unique solution to \eqref{eq_continu}-\eqref{bc_continu}. 

Under the assumption that $\varepsilon_0$ is sufficiently small and that   $a^3/d^3 < \varepsilon_0,$  for arbitary $p\in [1,3/2)$ there exists a constant $C_0$ depending only on $p,\varepsilon_0,K$ for which:
\begin{multline} \label{eq_mainresult}
\|u_N - u_c\|_{L^p_{loc}(\mathbb R^3)}\\
 \leq C_p(K,\varepsilon_0) |A|\left[ \|\mathbb M_N -\mathbb M_{eff}\|_{\dot{H}^{-1}(\RR^3)} + \left( \dfrac{a^3}{d^3}\right)^{1+\theta}   + \|\mathbb{M}_{eff}\|^2_{L^{\infty}(\mathbb R^3)} \right]
 \end{multline}
 where $\theta = \frac 1p - \frac 23.$
\end{theorem}
 
Several comments are in order. First, In \eqref{eq_mainresult},
the $\dot{H}^{-1}(\mathbb R^3)$ norm on the right-hand side must be understood componentwise. Second, in the particular case of spheres, we can 
compute $\mathbb M_{eff}$ via \eqref{eq_casesphere} so that, we obtain a fully rigorous justification of the system:
\begin{align} \label{eq_continu_sphere}
  & \left\{  
 \begin{aligned}
 - {\rm div} \left[2\mu\left(1+ 5\phi f \right) D(u)  - p \mathbb{I}_3\right] &=& 0 \\
 {\rm div} u &=& 0 
 \end{aligned}
 \right. && \text{ on $\mathbb R^3$}, \\ \label{bc_continu_sphere}
  & \quad u(x) = A x && \text{ at infinity},
 \end{align}
that has been obtained previously in \cite{Niet-Schub,Almog-Brenner}.
Finally, the restriction on exponent $p$ is reminiscent of the singularity 
of solutions to \eqref{eq_Stokeslocal}, corresponding to the gradient of the 
Green function for the Stokes problem on $\mathbb R^3,$ {\em i.e.} like $1/|x|^{2}$. This singularity allows an $L^p$-space for $p<3/2$ in 
dimension 3. In particular, this restriction can be removed when measuring 
the distance between $u_N$ and $u_c$ outside the particle domain $K$
(see \cite{Niet-Schub} in the case of spheres).
 
\medskip 
 
As in the original proof of Einstein, Theorem \ref{thm_main} relies on two main properties. First, 
each particle in the cloud behaves as if it was alone in the strain flow $x \mapsto Ax.$ Second, there is an
underlying additivity principle which implies that the action of the cloud of particles on the fluid 
is the sum of the undividual actions of the different particles.  In the two next sections, we justify the first of these two properties by extending Einstein computations to general suspensions. Broadly, a first guess $(u,p)$ for a solution to \eqref{eq_Stokes}-\eqref{bc_Stokes}-\eqref{eq_Newton} could be
\[
u_{app}^{(0)}(x) = Ax \qquad p_{app}^{(0)}(x) = 0. 
\]
This yields a solution to \eqref{eq_Stokes} and \eqref{eq_Newton} which does not fulfill  the boundary
conditions \eqref{bc_Stokes} on $\partial B_l.$ So, we apply the linearity of the Stokes problem and introduce a first corrector:
\[
\left\{
\begin{aligned}
u_1 (x) &= &  \sum_{l=1}^{N} U[A,B_l](x-x_l)  \\
p_1(x) & =&    \sum_{l=1}^{N} \mu P[A,B_l](x-x_l) 
\end{aligned}
\right.
 \qquad \text{ where }
\left\{
\begin{aligned}
  U[A,B_l](x) &=&  a U[A,\mathcal B_l]\left(\dfrac{x}{a}\right) \\
  P[A,B_l](x) &=& P[A,\mathcal B_l] \left(\dfrac{x}{a}\right)
 \end{aligned}
  \right.
\]
Again the candidate $u_{app} ^{(1)}= u^{(0)}_{app}+u_1$ is a solution to \eqref{eq_Stokes} and \eqref{eq_Newton}
but does not match boundary conditions \eqref{bc_Stokes}. So, we proceed with compensating again the
non rigid part of the velocity-field $u_{app}^{(1)}$ on the boundaries $\partial B_l.$ This starts a process known as the "method of reflections". It has been studied in other contexts in \cite{Velazquez,Salomon,OJ} and extended to the problem of effective viscosity for a suspension of spheres in \cite{Niet-Schub}. Herein, we modify a bit the method by  correcting only the first order term in the expansion of the boundary values of 
$u_{app}^{(k)}$ on $\partial B_l:$
\[
u_{app}^{(k)}(x) = V^{(k)}_l + \tilde{A}^{(k)}_{l} \cdot (x - x_l)+ O(|x-x_l|^2),
\qquad  V^{(k)}_l = u_{app}^{(k)}(x_l), \; \tilde{A}^{(k)}_l = \nabla u_{app}^{(k)}(x_l).
\]
This enables to rely on the semi-explicit solutions $(U[A,\mathcal B_l],P[A,\mathcal B_l])$ to \eqref{eq_Stokeslocal} 
and relate the final computations with the associated $\mathbb M[A,B_l].$ However, this does not 
rule out the key-difficulty of the process. Indeed, the method of reflections leads to the iterative
formula:
\[
A^{(k+1)}_l=  \sum_{\lambda \neq l}  D(U)[A^{(k)}_{\lambda},B_{\lambda}](x_{l}- x_{\lambda}).
\]
with a kernel $D(U)[A,B_{\lambda}]$ wich decays generically like $x \mapsto a^3A/|x|^3.$ A priori,
the above iterative formula entails then the bound:
\[
\max_{l} |A^{(k+1)}_{l}| \lesssim \left( \dfrac{a^3}{d^3}\right) |\ln(N)| \max_{l} |A^{(k)}_{l}|
\] 
which yields that $a^3/d^3$ must be small w.r.t. $\ln(N)$ for the method to converge (see  assumption (2.3) in  \cite{Niet-Schub}).  We remove this difficulty herein
by showing that there exists a Calder\`on-Zygmund operator underlying the above recursive formula. This enables to rule out the limitation on $a^3/d^3$ with respect to the  number $N$ of particles. 
These computations are explained in the two next sections. 
Section \ref{sec_cellpbm} is devoted to the analysis of the problem \eqref{eq_Stokeslocal}-\eqref{eq_Newtonlocal}. The Section \ref{sec_refmet} builds up on this analysis to study the convergence of the method of reflections and compute 
error estimates between the sequence of approximated solutions $u_{app}^{(n)}$ and the exact solution $u_N$ to \eqref{eq_Stokes}-\eqref{bc_Stokes}-\eqref{eq_Newton}.

\medskip

The two last sections are devoted to the proof of the additivity principle and to complete the proof of Theorem \ref{thm_main}. Once the method of reflections is proved to converge, we have an expansion of the solution 
to \eqref{eq_Stokes}-\eqref{bc_Stokes}-\eqref{eq_Newton} in terms of the parameter $a^3/d^3.$ We prove that there exits an equivalent expansion of the
solution to \eqref{eq_continu}-\eqref{bc_continu} w.r.t. $\mathbb M_{eff}$ so that there is a correspondance between the first terms in the expansions of both solutions.
We emphasize that, as classical with the weak-formulation of the Stokes problem, one obtains estimates on the difference of  velocity-fields 
$u_N-u_c.$  Regularity properties of the Stokes problem entail then 
similar properties for the pressures.

\medskip

Through the paper, we use the following conventions. In the space of $3\times 3$ matrices ${\rm Mat}_{3}(\mathbb R),$
we denote by ${\rm Sym}_3(\mathbb R)$ the set of symmetric matrices
and ${\rm Sym}_{3,\sigma}(\mathbb R)$ its subspace containing only the trace-free ones. We denote $\mathbb P_{3,\sigma}$ the orthogonal projection from 
 ${\rm Mat}_3(\mathbb R)$  onto ${\rm Sym}_{3,\sigma}(\mathbb R)$ with respect to the matrix contraction. 
 
Concerning function spaces, we use classical notations for Lebesgue and Sobolev spaces. We also introduce the Beppo-Levi space $\dot{H}^1$
and  its divergence-free variant:
\beq
\dot{H}^1_\sigma(\mathcal O):=\{u\in\dot{H}^1(\mathcal O) \text{ such that } {\rm div } u=0 \text{ on $\mathcal O$\}}.
\enq

In the whole paper, we denote $\mathbf{U}:=(U^{ij})_{1\leq i,j\leq 3}$ and $\mathbf{q}:=(q_1,q_2,q_3)$ the fundamental solution to the Stokes equation in the whole space $\RR^3$, which can be written 
\beq
U^{ij}:=-\frac{1}{8\pi}\left[\frac{\delta_{ij}}{|x-y|}+\frac{(x_i-y_i)(x_j-y_j)}{|x-y|^3}
\right],
\qquad
q_j=\frac{1}{4\pi}\frac{x_j-y_j}{|x-y|^3}.
\enq
for $i,j=1,2,3.$ We collect $(U^{i,1},U^{i,2},U^{i,3})$ in the vector $\mathcal U^{i}.$

We also introduce the Bogovskii operator $\mathfrak{B}_{\mathcal{B}}[f]$ 
defined for arbitrary mean-free $f\in L^2(\mathcal{B})$. It is well-known that this $\mathfrak{B}_{\mathcal{B}}$ is continuous with values in $H_0^1(\mathcal{B})$ and characterized by ${\rm div} \mathcal B_{\mathcal{B}}[f]=f$ in $\mathcal{B}$. In particular, we denote ${\rm div} \mathfrak{B}_{\lambda_1,\lambda_2}[f]:=\mathfrak{B}_{B(0,\lambda_2)\setminus \overline{B(0,\lambda_1)}}[f]$  for any $0<\lambda_1<\lambda_2$.

\medskip


\section{Analysis of the Stokes problem} \label{sec_cellpbm}
In the whole section, we suppose that $\mathcal{B}\subset B(0,1)\subset\RR^3$ has smooth boundaries $\p\mathcal{B}.$ Given a trace-free  $A \in {\rm Sym}_{3,\sigma}(\mathbb R),$ let consider the following problem:
\begin{align} \label{eq_Stokessec}
& \left\{
\begin{aligned}
- {\rm div \Sigma(u,p)}&=& 0  && \text{ in $\mathbb R^3 \setminus \overline{\mathcal B}$},\\
{\rm div} u &=& 0 &&\text{ in $\mathbb R^3 \setminus \overline{\mathcal B}$},
\end{aligned}
\right.
&& 
\left\{
\begin{aligned}
&u(x) = -Ax + V + \omega \times x && \text{ on $\partial \mathcal B$}, \\
&u(x) = 0 && \text{ at infinity },
\end{aligned}
\right. 
\\ \label{eq_Newtsec}
& \int_{\partial \mathcal B} \Sigma(u,p)n{\rm d}\sigma = 0 &&
\int_{\partial \mathcal B} x \times \Sigma(u,p)n{\rm d}\sigma = 0. 
\end{align}
It is classical that, given alternatively a $3\times3$ matrix $A$  and  $V,\omega \in \mathbb R^3 \times \mathbb R^3,$ there exists a unique solution $(u[A],p[A])$ (with $V=\omega=0$)
and $(u[V,\omega],p[V,\omega])$ (with $A=0$) to \eqref{eq_Stokessec}  in $\dot{H}^1(\mathbb R^3) \times L^2(\mathbb R^3)$. The mapping 
\[
(V,\omega) \mapsto ( \int_{\partial \mathcal B} \Sigma(u[V,\omega],p[V,\omega])n{\rm d}\sigma ,
\int_{\partial \mathcal B} x \times \Sigma(u[V,\omega],p[V,\omega])n{\rm d}\sigma) 
\]
is then linear and symmetric positive definite. In particular, 
there exists a unique solution $(V_A,\omega_A)$ to the problem:
\begin{align*}
 \int_{\partial \mathcal B} \Sigma(u[V_A,\omega_A],p[V_A,\omega_A])n{\rm d}\sigma & =-\int_{\partial \mathcal B} \Sigma(u[A],p[A])n{\rm d}\sigma \\
\int_{\partial \mathcal B} x \times \Sigma(u[V_A,\omega_A],p[V_A,\omega_A])n{\rm d}\sigma
&=- \int_{\partial \mathcal B} x \times \Sigma(u[A],p[A])n{\rm d}\sigma.
\end{align*}
The candidate  $U[A,\mathcal B]= u[A] + u[V_A,\omega_A], \; P[A,\mathcal B] = p[A] +p[V_A,\omega_A]$ is then a solution to \eqref{eq_Stokessec}-\eqref{eq_Newtsec}. By difference and integration by parts, we obtain uniqueness of a velocity-field solution  which enables to recover that the pressure is unique up to a constant also. 
Since $U[A,\mathcal B]$ matches a velocity-field of the form $-Ax + V + \omega\times x$ on $\partial \mathcal B,$ it is classical to extend $U[A,\mathcal B]$
by the field corresponding to this boundary value on $\mathcal B$ yielding 
a vector-field $U[A,\mathcal B] \in \dot{H}^1_{\sigma}(\mathbb R^3).$   
Straightforward integration by parts arguments show that
this extended $U[A,\mathcal B]$ realizes:
\[
\min \Bigl \{ \int_{\mathbb R^3} |D(u)|^2, \quad u \in \dot{H}^1_{\sigma}(\mathbb R^3), \quad  D(u) = -A \text{ on $\mathcal B$}\Bigr\}.
\] 
In particular, we note that the set on which the minimum is computed on the right-hand side increases when $\mathcal B$ decreases. Since we assume 
$\mathcal B \subset B(0,1)$ in this section, we infer a uniform bound for $\|D(U[A,\mathcal B])\|_{L^{2}(\mathbb R^3)}$ by 
the minimum reached for $\mathcal B = B(0,1).$ This yields that
\begin{equation} \label{eq_unifbound}
\|D(U[A,\mathcal B])\|_{L^2(\mathbb R^3)}
\leq C|A|
\end{equation}
(and thus $\|U[A,\mathcal B]\|_{\dot{H}^1(\mathbb R^3)} \leq C|A|$ also) with  a constant 
$C$ uniform in $\mathcal B \subset B(0,1).$ 

One may proceed similarly to show that, under the assumption \eqref{H1}-\eqref{H2}, the problem \eqref{eq_Stokes}-\eqref{bc_Stokes}-\eqref{eq_Newton} admits a unique solution $(u_N,p_N)$ such that 
\begin{equation} \label{eq_integrabilite}
v_N : x \mapsto u_N(x)- Ax \in \dot{H}^1 (\mathcal F_N), \qquad x \mapsto p_N(x) \in L^2(\mathcal F_N)
\end{equation}
Furthermore, the velocity-field of this solution can be extended to the whole $\mathbb R^3$ to yield a vector-field $v_N(x) = u_N(x) - Ax$ that realizes:
\[
\min \Bigl \{ \int_{\mathbb R^3} |D(u)|^2, \quad u \in \dot{H}^1_{\sigma}(\mathbb R^3), \quad  D(u) = -A \text{ on $\mathcal B_l  \quad \forall \, l=1,\ldots,N$}\Bigr\}.
\]
In particular, under assumptions \eqref{H1}-\eqref{H2} we can construct an extension $\tilde{v}_N$ on $\bigcup_{l} B(x_l,2a)$ of the field that matches $x \mapsto -A(x-x_l)$ on each of the $B(x_l,a)$ (and thus on $B_l \subset B(x_l,a)$) by truncating and lifting  the divergence terms. 
Straightforward computations show that we have then:
\[
\int_{\mathbb R^3} |\nabla v_N|^2 = 2 \int_{\mathbb R^3} |D(v_N)|^2 
\leq  2 \int_{\mathbb R^3} |D(\tilde{v}_N)|^2 \leq C \dfrac{a^3}{d^3} |A|.
\]
so that there exists a uniform constant $C$ for which:
\begin{equation} \label{eq_unifuN}
\|\nabla v_N\|_{L^2(\mathbb R^3)} \leq C \left( \dfrac{a^3}{d^3}\right)^{\frac 12}.
\end{equation}

\medskip

Before going to the main result of this section, we prepare the proof with a 
control on momenta of the trace of 
\[
\Sigma[A,\mathcal{B}] = 2 D(U)[A,\mathcal{B}] - P[A,\mathcal{B}] \mathbb{I}_3.
\]
on $\partial \mathcal B.$ This is the content of the following preliminary lemma: 
\begin{lemma}\label{resistence}

There exists an absolute constant $C$ such that:
\beq
\Bigl| \mathbb P_{3,\sigma}\left[ \int_{\p\mathcal{B}}\Sigma[A,\mathcal{B}]n\otimes yd\sigma(y)\right] \Bigr|\leq C|A|,
\enq

\end{lemma}

We recall that $\mathbb P_{3,\sigma}$ stands for the orthogonal projection from ${\rm Mat}_{3}(\mathbb R)$ onto ${\rm Sym}_{3,\sigma}(\mathbb R).$
With this lemma, we obtain that the linear mappings 
\[
A  \mapsto  \mathbb P_{3,\sigma}\left[ \int_{\p\mathcal{B}}\Sigma[A,\mathcal{B}]n\otimes yd\sigma(y)\right]
\]
are uniformly bounded whatever $\mathcal B \subset B(0,1).$

\begin{proof}
Because of the linearity of the Stokes equations and of the stress tensor, the mapping from ${\rm Sym}_{3,\sigma}(\mathbb R)$ to ${\rm Sym}_{3,\sigma}(\mathbb R)$:
\beq
\mathcal{L}:~~A\mapsto \mathbb P_{3,\sigma}\left[ \int_{\p\mathcal{B}}\Sigma[A,\mathcal{B}]n\otimes yd\sigma(y)\right]
\enq
is also linear. So, let $(\mathcal{E}_i)_{i=1,\dots,5}$ an orthonormal basis of ${\rm Sym}_{3,\sigma}(\mathbb R)$, and $(V_i:=U[\mathcal{E}_i,\mathcal{B}])_{i=1,\ldots,5}$ the corresponding velocity-fields solution to the Stokes problem \eqref{eq_Stokessec}-\eqref{eq_Newtsec}. Then, the mapping $\mathcal{L}$ is represented  in this basis by the matrix $\mathbb{L}$:
\beq
\mathbb{L}:=\Big(\int_{\p\mathcal{B}}\Sigma[\mathcal{E}_i,\mathcal{B}]n\otimes yd\sigma(y):\mathcal{E}_j\Big)_{1\leq i,j\leq 5},
\enq 
Our proof reduces to obtaining that $|\mathbb{L}_{i,j}| \leq C$
for arbitrary $i,j$ in $\{1,\ldots,5\}.$  So, let fix $i,j,$ by integrating by parts, we have that 
\[
 \mathbb{L}_{i,j} = \int_{\partial \mathcal B} \Sigma[\mathcal E_i,\mathcal B]n \cdot (\mathcal E_j y){\rm d}\sigma(y)
 = 
 2\int_{\RR^3\setminus\overline{\mathcal{B}}}D(V_i):D(V_j)  .
\]
so that:
\[
|\mathbb L_{i,j}| \leq 2 \|D(V_i)\|_{L^2(\RR^3\setminus\overline{\mathcal{B}})}\|D(V_j)\|_{L^2(\RR^3\setminus\mathcal{B})} \leq C|\mathcal E_i||\mathcal E_j|,
\]
where we applied \eqref{eq_unifbound} to obtain the last inequality.
This concludes the proof. 
\end{proof}

\medskip

We continue the analysis of \eqref{eq_Stokessec}-\eqref{eq_Newtsec} 
by providing pointwise estimates on $U[A,\mathcal B]$. The content of the following theorem is reminiscent of \cite[Section V.3]{Galdi}:
\begin{proposition}\label{asym-stokes} 
Let $(U[A,\mathcal{B}],P[A,\mathcal B])$ be the unique solution to \eqref{eq_Stokessec}-\eqref{eq_Newtsec}. There exists a vector field $\mathcal{H}[A,\mathcal{B}]$ depending on $A$ and $\mathcal{B}$, such that for any $|x|>4$, 
\beq
U[A,\mathcal{B}](x)=\mathcal{K}[A,\mathcal{B}](x)+\mathcal{H}[A,\mathcal{B}](x),
\enq
where $\mathcal{K}[A,\mathcal{B}]_i(x)=\mathbb{M}[A,\mathcal{B}]:\nabla\mathcal{U}^{i}(x)$ for $i=1,2,3$ with:
\begin{equation} \label{eq_defM}
\mathbb M[A,\mathcal B] = \mathbb{P}_{3,\sigma}\left\{  \int_{\p\mathcal{B}}\left[ -(\Sigma[A,\mathcal{B}]n)\otimes y + 2U[A,\mathcal B] \otimes n \right]d\sigma(y)\right\}.
\end{equation}
 Moreover, there exists a constant 
$C$ independent of $A$ for which:
\beq
|\mathbb M[A,\mathcal B]| \leq C|A|\,, \qquad
|\nabla^{\beta}\mathcal{H}[A,\mathcal{B}](x)|\leq C|A|\frac{1}{|x|^{3+|\beta|}}
\quad \forall \, \beta \in \mathbb N^3.
\enq
\end{proposition}

Before giving a proof of this proposition, we note that for 
large $x$, we have:
\[
\nabla \mathcal{U}^i(x) \sim \dfrac{C}{|x|^2}.
\]
Consequently the splitting that we obtain in the above proposition
corresponds to the extraction of the leading order term ($\mathcal K[A,\mathcal B](x)$) at infinity. A second crucial remark induced by this proposition is that the amplitude of both terms (the leading term $\mathcal K$ and remainder $\mathcal H$) do not depend asymptotically on the shape $\mathcal B.$

\begin{proof}
 Let $\chi(x)\in C_0^{\infty}(\mathbb R^3)$ such that $\chi\equiv1$ in $\RR^3\setminus B(0,2)$ and $\chi\equiv0$ in $B(0,1)$. We recall that $\mathfrak{B}_{\lambda_1,\lambda_2}$ stands for the Bogovskii operator lifting the divergence  on the annulus $B(0,\lambda_2)\setminus \overline{B(0,\lambda_1)}$. 
 
 \medskip
 
By standard ellipticity arguments $U[A,\mathcal B],P[A,\mathcal B]$ are $C^{\infty}$ on $\mathbb R^3 \setminus \overline{\mathcal B}.$ Let define
\begin{align*}
\bar{U}[A,\mathcal{B}](x) & :=U[A,\mathcal{B}](x)\chi(x)-\mathfrak{B}_{1,2}[U[A,\mathcal{B}] \cdot \nabla \chi](x) \\
\bar{P}[A,\mathcal{B}](x) & :=P[A,\mathcal{B}](x)\chi(x).
\end{align*}
Up to a mollifying argument that we skip for conciseness, we may assume that $\bar{U}[A,\mathcal B]
\in C^{\infty}(\mathbb R^3).$ The pair $(\bar{U}[A,\mathcal B],\bar{P}[A,\mathcal{B}])$ satisfies then the Stokes equation on $\RR^3$ with source term $f_\chi[A,\mathcal{B}]  = -{\rm div} \, \bar{\Sigma}[A,\mathcal{B}]  \in C^{\infty}_c(B(0,2)\setminus \overline{B(0,1)})$ where:
\[
\bar{\Sigma}[A,\mathcal{B}]:=-\bar{P}[A,\mathcal{B}]\mathbb{I}_3+2D(\bar{U}[A,\mathcal{B}]).
\]
Since $\bar{U}[A,\mathcal B] \in  \dot{H}^1(\mathbb R^3)$ and we have uniqueness of $\dot{H}^1$-solutions to the Stokes equations on $\mathbb R^3,$ we may use the Green function $\mathbf{U}$  to compute $\bar{U}[A,\mathcal B]$. This entails that, for each $i=1,2,3,$
we have:
\beq
\bar{U}[A,\mathcal{B}]_i=\int_{\RR^3}f_\chi[A,\mathcal{B}](y)\cdot\mathcal{U}^i(x-y)dy.
\enq
In particular, for $|x|>2$ and $i=1,2,3$ (where $\bar{U}[A,\mathcal{B}]$ coincides with $U[A,\mathcal{B}]$), a Taylor expansion yields:
\beq
U[A,\mathcal{B}]_i(x)&=&\int_{\RR^3}f_\chi[A,\mathcal{B}](y)dy\cdot\mathcal{U}^i(x)-\int_{\RR^3}f_{\chi}[A,\mathcal{B}](y)\otimes ydy:\nabla\mathcal{U}^i(x)\\
&&+\sum_{|\alpha|=2}\int_{\RR^3}f_\chi[A,\mathcal{B}](y)\cdot\int_0^1(1-t)y^{\alpha}D^{\alpha}\mathcal{U}^{i}(x-ty)dtdy.
\\
&=& T_0 + \mathcal K[A,\mathcal B](x)+ \mathcal H[A,\mathcal B](x).
\enq
Concerning $T_0$, we notice that
\beq
\int_{\RR^3}f_{\chi}[A,\mathcal{B}](y)dy&=& \int_{B(0,3)\setminus B(0,1)}
{\rm div} \bar{\Sigma}[A,\mathcal{B}] \\
&=&\int_{\p B(0,3)}\bar{\Sigma}[A,\mathcal{B}](y)nd\sigma(y)=\int_{\p B(0,3)}\Sigma[A,\mathcal{B}](y)nd\sigma(y)\\
&=&\int_{\p\mathcal{B}}\Sigma[A,\mathcal{B}](y)nd\sigma(y) =0
\enq
Hence $T_0=0$. To analyse $\mathcal K[A,\mathcal B](x)$, we denote:
\beq
{\mathcal{M}}[A,\mathcal{B}]:=\frac{1}{2}\Bigl(\int_{\RR^3}f_{\chi}[A,\mathcal{B}](y)\otimes ydy+(\int_{\RR^3}f_{\chi}[A,\mathcal{B}](y)\otimes ydy)^T\Bigr)
\enq
and
\beq
\mathcal{A}[A,\mathcal{B}]:=\frac{1}{2}\Bigl(\int_{\RR^3}f_{\chi}[A,\mathcal{B}](y)\otimes ydy-(\int_{\RR^3}f_{\chi}[A,\mathcal{B}](y)\otimes ydy)^T\Bigr).
\enq
First, for arbitrary skew-symmetric matrix $E$, there holds:
\beq
\int_{\RR^3}f_{\chi}[A,\mathcal{B}](y)\cdot(Ey)dy&=&\int_{B(0,3)\setminus B(0,1)}f_{\chi}[A,\mathcal{B}](y)\cdot(Ey)dy\\
&=&\int_{\p B(0,3)}\big[\bar{\Sigma}[A,\mathcal{B}]n\big]\cdot Eyd\sigma(y)-\int_{B(0,3)\setminus B(0,1)}\bar{\Sigma}[A,\mathcal{B}]:\nabla (Ey)dy,
\enq
On the right-hand side, we have, since $\bar{\Sigma}[A,\mathcal{B}]$ 
is symmetric and $E$ is skew-symmetric:
\beq
\int_{B(0,3)\setminus B(0,1)}\bar{\Sigma}[A,\mathcal{B}]:\nabla (Ey)dy=\int_{B(0,3)\setminus B(0,1)}\bar{\Sigma}[A,\mathcal{B}]: E\,  dy=0.
\enq
We also notice that 
\beq
\int_{\p B(0,3)}\big[\bar{\Sigma}[A,\mathcal{B}]n\big]\cdot Eyd\sigma(y)&=&\int_{\p B(0,3)}\big[\Sigma[A,\mathcal{B}]n\big]\cdot Eyd\sigma(y)\\
&=&\int_{B(0,3)\setminus B(0,1)}\Sigma[A,\mathcal{B}]:\nabla (Ey)dy - \int_{\p\mathcal{B}}\big[\Sigma[A,\mathcal{B}]n\big]\cdot Eyd\sigma(y)\\
&=& - \int_{\p\mathcal{B}}\big[\Sigma[A,\mathcal{B}]n\big]\cdot Eyd\sigma(y) = 0.
\enq
To obtain the last equality, we use that since $E$ is skew-symmetric, there is a vector $e$ such that $Ex=e\times x$ and:
\beq
\int_{\p\mathcal{B}}\big[\Sigma[A,\mathcal{B}]n\big]\cdot Eyd\sigma(y)=\int_{\p \mathcal{B}}\Sigma[A,\mathcal{B}] n\times yd\sigma(y)\cdot e=0.
\enq
Therefore we obtain that $\mathcal{A}[A,\mathcal{B}]=0$. Consequently, we 
have 
\[
\mathcal K[A,\mathcal B](x)= {\mathcal M}[A,\mathcal B] : \nabla \mathcal U^i(x).
\]
Since 
${\mathcal M}[A,\mathcal B]$ is symmetric, we deduce that:
\[
\mathcal K[A,\mathcal B](x)= {\mathcal M}[A,\mathcal B] : D(\mathcal U^i)(x) = \mathbb P_{3,\sigma}  [{\mathcal M}[A,\mathcal B] ] :  D(\mathcal U^i)(x) 
= \mathbb P_{3,\sigma}[{\mathcal M}[A,\mathcal B]]  :  \nabla \mathcal U^i(x). 
\]
So, we set $\mathbb M[A,\mathcal B] =\mathbb P_{3,\sigma}[{\mathcal M}[A,\mathcal B]]$ and we turn to show  \eqref{eq_defM} and $|\mathbb{M}[A,\mathcal{B}]|\leq C|A|$. To this end, we notice  that $\mathbb{M}[A,\mathcal B]$ is completely fixed by its action on matrices $S \in {\rm Sym}_{3,\sigma}(\mathbb R).$ 
So, let fix $S \in {\rm Sym}_{3,\sigma}(\mathbb R).$ We have
\beq
\int_{\RR^3}f_{\chi}[A,\mathcal{B}](y)\otimes y dy:S=\int_{\RR^3}f_{\chi}[A,\mathcal{B}](y)\cdot(Sy)dy=\int_{B(0,3)\setminus \mathcal{B}}f_{\chi}[A,\mathcal{B}](y)\cdot(Sy)dy.
\enq
Applying that $\mathrm{div}\bar{\Sigma}[A,\mathcal{B}]=f_\chi[A,\mathcal{B}]$ again, we obtain
\beq
\int_{B(0,3)\setminus \mathcal{B}}f_{\chi}[A,\mathcal{B}](y)\cdot(Sy)dy&=&\int_{\p B(0,3)}(\bar{\Sigma}[A,\mathcal{B}]n)\cdot(Sy)d\sigma(y)-\int_{B(0,3)\setminus \mathcal{B}}\bar{\Sigma}[A,\mathcal{B}]:Sdy\\
&=&\int_{\p B(0,3)}(\bar{\Sigma}[A,\mathcal{B}]n)\cdot(Sy)d\sigma(y)-2\int_{B(0,3)\setminus \mathcal{B}}D(\bar{U}[A,\mathcal{B}]):Sdy\\
&&+\int_{B(0,3)\setminus \mathcal{B}}\bar{P}[A,\mathcal{B}]\mathbb{I}_3:Sdy
\enq
Since $S$ is trace-free, the last pressure term vanishes.
We rewrite the first term on the right-hand side: 
\beq
\int_{\p B(0,3)}(\bar{\Sigma}[A,\mathcal{B}]n)\cdot(Sy)d\sigma(y)&=& \int_{\p B(0,3)}(\Sigma[A,\mathcal{B}]n)\cdot(Sy)d\sigma(y)\\
&=& - \int_{\p\mathcal{B}}(\Sigma[A,\mathcal{B}]n)\cdot(Sy)d\sigma(y)+2\int_{B(0,3)\setminus\mathcal{B}}D(U[A,\mathcal{B}]):Sdy
\enq
 This entails that:
\beq
\int_{B(0,3)\setminus \mathcal{B}}f_{\chi}[A,\mathcal{B}](y)\cdot(Sy)dy&=& - \int_{\p\mathcal{B}}(\Sigma[A,\mathcal{B}]n)\cdot(Sy)d\sigma(y)\\
& &+2 \int_{B(0,3)\setminus\mathcal{B}}D(U[A,\mathcal{B}]-\bar{U}[A,\mathcal{B}]):Sdy.
\enq
We recall that pressure term do vanish since $S$ is trace-free.
Concerning the first integral on the right-hand side, we notice again that: 
\beq
\int_{\p\mathcal{B}}(\Sigma[A,\mathcal{B}]n)\cdot(Sy)d\sigma(y)= \mathbb P_{3,\sigma}\left[ \int_{\p\mathcal{B}}(\Sigma[A,\mathcal{B}]n)\otimes yd\sigma(y)\right]:S.
\enq
We are then in position to apply Lemma \ref{resistence} which yields that
\begin{equation} \label{eq_controlbord}
\Big|\int_{\p\mathcal{B}}(\Sigma[A,\mathcal{B}]n)\cdot(Sy)d\sigma(y)\Big|\leq C|A||S|,
\end{equation}
for an absolute constant $C.$
We proceed with the second integral. We notice that $U[A,\mathcal{B}](x)=\bar{U}[A,\mathcal{B}](x)$ for any $|x|>2$ and $\bar{U}[A,\mathcal{B}](x)=0$ on $\p\mathcal{B}$. Hence
\beq
\int_{B(0,3)\setminus\mathcal{B}}D(U[A,\mathcal{B}]-\bar{U}[A,\mathcal{B}]):Sdy&=&\int_{B(0,3)\setminus\mathcal{B}}D\big[U[A,\mathcal{B}]-\bar{U}[A,\mathcal{B}]\big]dy:S\\
&=& \int_{\p B(0,3) \cup \p \mathcal B} (U[A,\mathcal{B}]-\bar{U}[A,\mathcal{B}]) \otimes n  d\sigma  : S  \\
&=& \int_{\p \mathcal B} U[A,\mathcal{B}] \otimes n  d\sigma : S\\
&=& A:S\ |\mathcal{B}|,
\enq
where we applied that $U[A,\mathcal B](y) = Ay + V + \omega \times y $ 
to obtain the last identity. Gathering the previous computations 
we obtain that, for arbitrary $S \in {\rm Sym}_{3,\sigma}(\mathbb R)$, 
there holds:
\begin{align*}
\mathbb M[A,\mathcal B] : S & = \int_{\mathbb R^3} f_{\chi}[A,\mathcal B](y) \cdot S y dy \\
& = - \mathbb{P}_{3,\sigma} \left[\int_{\partial \mathcal B} \Sigma[A,\mathcal B] n \otimes y d\sigma(y) \right] : S +2  \int_{\partial B} U[A,\mathcal B](y) \otimes nd\sigma(y) : S.
\end{align*}
This concludes the proof of \eqref{eq_defM}.  Recalling that $\mathcal{B} \subset B(0,1)$, and applying the explicit computation of the last integral 
in this latter identity, we obtain also $|\mathbb{M}[A,\mathcal{B}]|\leq C|A|.$ 

\medskip 

To finish the proof, we handle the last term $\mathcal H[A,\mathcal B](x)$ for $|x|>4$. We prove the required estimate for $\beta =(0,0,0),$ the extension to arbitrary $\beta$ being obvious.  
Given $|\alpha|=2$, and $|x|>2$, we can find $\phi_x\in C^{\infty}(\RR^3)$
with $\mathrm{Supp}(\phi_x)\subset \overline{B(0,3)}\setminus B(0,1)$
such that:
\beq
\int_0^1(1-t)y^{\alpha}D^{\alpha}\mathcal{U}^{i}(x+ty)dt =:\phi_x(y) 
\quad 
\forall \, y \in \mathrm{Supp}(\chi) 
\enq
Asymptotic properties of $\mathcal U^i$ entail then that
$\|\phi_x\|_{W^{1,\infty}(\mathbb R^3)} \leq C/|x|^3.$  
Therefore, we have, thanks to the uniform bound \eqref{eq_unifbound}
and the embedding $\dot{H}^1(\mathbb R^3) \subset L^2_{loc}(\mathbb R^3):$
\begin{multline*}
\Big|\int_{\RR^3}f_\chi[A,\mathcal{B}](y)\cdot\int_0^1(1-t)y^{\alpha}D^{\alpha}\mathcal{U}^{i}(x+ty)dtdy\Big| \\
\begin{array}{l}
 \leq C\|f_\chi[A,\mathcal{B}]\|_{\dot{H}^{-1}(B(0,3)\setminus \overline{B(0,1)})}\|\phi_x\|_{{H}^1_0(B(0,3) \setminus \overline{B(0,1)})} \\[8pt]
\leq C\|\bar{U}[A,\mathcal{B}]\|_{\dot{H}^1(B(0,3)\setminus \overline{B(0,1)})}\frac{1}{|x|^3}\leq C\|U[A,\mathcal{B}]\|_{\dot{H}^1(\RR^3)}\frac{1}{|x|^3}\leq C \frac{|A|}{|x|^3}.
\end{array}
\end{multline*}
This ends the proof of the proposition.
\end{proof}


We end this section by having a look to the interactions between the decomposition in Proposition \ref{asym-stokes} with scaling properties of the Stokes problem \eqref{eq_Stokessec}-\eqref{eq_Newtsec}. Indeed, for any  $a < 1$,
standard scaling arguments imply that:
\[
U[A,a \mathcal{B}](x)=aU[A,\mathcal{B}](x/a),
\qquad 
P[A,a\mathcal{B}](x)=P[A,\mathcal{B}](x/a).
\]
Consequently for arbitrary $w \in \mathbb S^2,$ we have:
\[
\lim_{t \to \infty} t^2 U[A,a \mathcal{B}](tw) =a^3 \lim_{t\to \infty} t^2 U[A,\mathcal{B}](tw)
\] 
This entails that $\mathbb M[A,a\mathcal B]= a^3\mathbb M[A,\mathcal B]$  and 
\begin{equation} \label{eq_scalelead}
\mathcal{K}[A,a\mathcal{B}](x)=a^3\mathcal{K}[A,\mathcal{B}](x).\end{equation}
We can then compare remainder terms. This yields:
\begin{equation} \label{eq_scaleremainder}
|\nabla^{\beta} \mathcal{H}[A,a \mathcal{B}](x)| \leq \dfrac{C|A|a^4}{|x|^{3+|\beta|}}
\quad 
\forall \, |x| > 4a.
\end{equation}

\section{Approximation of the solution to the $N$-particle problem}
\label{sec_refmet}

In this section, we fix $N$ large and $A \in {\rm Sym}_{3,\sigma}(\mathbb R).$ 
We provide an approximation  {\em via} the method of reflections for the solution $(u_N,p_N)$ to
\begin{align}
\label{eq_Stokes_ref}
& \left\{
\begin{aligned}
- {\rm div}\, \Sigma_{\mu}(u,p) &= 0 \\
{\rm div}\, u &= 0 
\end{aligned}
\right.
\qquad \qquad
\text{ in $\mathbb R^3 \setminus \bigcup_{l=1}^N \overline{B}_l,$}
\\
\label{bc_Stokes_ref}
& \left\{
\begin{aligned}
& \; u(x) = V_l + \omega_l \times (x-x_l) && \text{ on $\partial B_l,$ for $l=1,\ldots,N$} \\ 
& \; u(x) = Ax && \text{ at infinity,}
\end{aligned}
\right.
\\
\label{eq_Newton_ref}
& \left\{
\begin{aligned}
\int_{\partial B_l} \Sigma_{\mu}(u,p) n {\rm d}s &= 0 \\
\int_{\partial B_l} (x-x_l) \times \Sigma_{\mu}(u,p) n {\rm d}s &= 0 \\
\end{aligned}
\right.
\qquad \qquad  \text{for }\, l = 1,\ldots,N.
\end{align}
We recall that the method of reflections consists in matching the boundary conditions on each particle by solving a Stokes system around each particle, gluing together the local solutions into one approximation and iterating the process, since by gluing the local solutions we alter the boundary values of the approximation. More precisely, we first define 
\[
u_{app}^{(0)}(x):=Ax, \qquad A^{(0)}_l:=A \text{ for $l=1,\dots,N$.}
\]
Given $n \geq 0$ and assuming that a vector-field $u_{app}^{(n)}$ and 
matrices $(A_{l}^{(n)})_{l=1,\ldots,N}$ are constructed  we set:
\begin{align} \label{eq_iter_matrix}
A^{(n+1)}_l &:= \sum_{l\neq\lambda}D(\mathcal{K})[A^{(n)}_\lambda,{B}_\lambda](x_l-x_\lambda)  && \forall \, l=1,\ldots,N \\
v^{(n+1)}(x) & :=  \sum_{l=1}^N U[A^{(n)}_{l},{B}_l](x-x_l) && 
\text{ on $\mathcal F^N$}  \label{eq_iter_corrector}\\
u_{app}^{(n+1)}(x)&:=  u_{app}^{(n)}(x)+v^{(n+1)}(x) && 
\text{ on $\mathcal F^N$}. \label{eq_iter_uapp}
\end{align}
Correspondingly, we compute a sequence of approximate pressure:
\[
p_{app}^{(0)}  = 0, \qquad
p_{app}^{(n+1)} = p_{app}^{(n)} + \mu \sum_{l=1}^{N} P[A^{(n)}_{l},{B}_l] \,, 
\qquad  
\forall \, n \in \mathbb N.
\]
The factor $\mu$ is introduced here since $(U,P)$ solves a Stokes system without viscosity.

\medskip

The motivation of these definitions is the following remark. For each $n \geq 0$
the flow $v^{(n+1)}$ cancels the first order symmetric term of the leading part of  the boundary value of $u_{app}^{(n)}$ on each $\partial B_l$. For instance, for $n=0$, we notice that on $\partial B_l$, it holds:
\beq
&&u^{(0)}(x)=Ax=Ax_l+A(x-x_l),\\
&&v^{(1)}(x)=-A(x-x_l) + \sum_{l\neq\lambda} U[A,{B}_l](x-x_\lambda),
\enq
which implies that:
\beq
u^{(1)}=Ax_l + \sum_{l\neq\lambda} U[A,{B}_l](x-x_\lambda).
\enq
Since $\varepsilon_0$ is very small, meaning that $a\ll d$, by Proposition \ref{asym-stokes} and \eqref{eq_scalelead}, we have that for each $\lambda\neq l$ and any $x\in\partial B_l$:
\beq
U[A,{B}_l](x-x_\lambda)&=&a^3\mathcal{K}[A,\mathcal{B}_\lambda](x-x_\lambda)+\mathcal{H}[A,{B}_\lambda](x-x_\lambda).
\enq
where $\mathcal{H}[A,{B}_\lambda](x-x_\lambda) << \mathcal{K}[A,B_\lambda](x-x_\lambda)$ since $|x-x_l| >> a$ on $\partial B_l.$ 
On the other hand, by Taylor expansion, for any  $x\in\partial B_l$ and any $\lambda\neq l$,
\beq
a^3\mathcal{K}[A,\mathcal{B}_\lambda](x-x_\lambda)=\mathrm{constant}+\mathrm{rotation}+a^3D(\mathcal{K})[A,\mathcal{B}_\lambda](x_l-x_\lambda)(x-x_l)+O(|x-x_l|^2).
\enq
Hence in the reflection method, we aim at canceling the symmetric gradient $a^3D\mathcal{K}[A,\mathcal{B}_\lambda](x_l-x_\lambda)(x-x_l)$. 
By a direct iteration,  we obtain that, for any $x\in\partial B_l$ and $n\geq 2,$
there holds:
\begin{align} \label{eq_splitboundary}
u^{(n)}_{app}(x) =&\mathrm{constant}+\mathrm{rotation} \\
&+a^3 \sum_{l\neq \lambda} \mathcal{K}[A^{(n-1)}_\lambda,\mathcal{B}_\lambda](x-x_\lambda) \notag \\
& + a^3 \sum_{j=0}^{n-2}
\sum_{l\neq\lambda}\Big( \mathcal{K}[A^{(j)}_\lambda,\mathcal{B}_\lambda](x-x_\lambda) - \mathcal{K}[A^{(j)}_\lambda,\mathcal{B}_\lambda](x_l-x_\lambda) 
\notag \\
& \qquad - [\nabla \mathcal{K}[A^{(j)}_\lambda,\mathcal{B}_\lambda](x_l-x_\lambda)](x-x_l)\Big) \notag \\
& +\sum_{j=0}^{n-1}\sum_{l\neq\lambda}\mathcal{H}[A^{(j)}_\lambda,{B}_\lambda](x-x_\lambda)\notag 
\end{align}
The purpose of this section is twofold. First, we show that the method of reflections converges. We quantify then how close the family of approximations
 $(u_{app}^{(n)})_{n\in \mathbb N}$ are  to the velocity-field  $u_N$
 solution to \eqref{eq_Stokes_ref}-\eqref{bc_Stokes_ref}-\eqref{eq_Newton_ref}.

\medskip

We start with the convergence of the method. Since the correctors are fixed
with respect to the family of matrices $(A_{l}^{(n)})_{l=1,\ldots,N}$ this amounts
to prove that this family of matrices defines a converging series (in $n$ for arbitrary $l\in \{1,\ldots,N\}$). This is the content of the following proposition 
which relies mostly on item (1) of Lemma \ref{modified goal} in Appendix \ref{app_ref}:
\begin{proposition}\label{Al}
There exists $\varepsilon_0 >0$ sufficiently small such that, for $\varepsilon < \varepsilon_0$ and $1<q<\infty$, there exists a constant $C(q,\varepsilon_0)$ depending on $q$ and $\varepsilon_0$, but independent of $N$, such that 
\beq
\Big(\sum_{l=1}^N |A^{(n+1)}_{l}|^q\Big)^{1/q}\leq C(q,\varepsilon_0)\left(\frac{a}{d}\right)^{3-3/q}\Big(\sum_{l=1}^N |A^{(n)}_{l}|^q\Big)^{1/q}
\quad \forall \, n \in \mathbb N.
\enq
\end{proposition}
\begin{proof}
Let $n \geq 0.$ We first notice that, by Proposition \ref{asym-stokes}, there exists symmetric matrices $\mathbb M^{(n)}_l := \mathbb M[A^{(n)}_{l},\mathcal B_l]$ such that 
\[
\mathcal K[A^{(n)}_{l},B_l]_i = a^3 \mathbb M^{(n)}_l : \nabla \mathcal U^i = 
a^3 \sum_{k=1}^3 [\mathbb M^{(n)}_l ]_{kl} \partial_{k} U^{i,l}, \quad
i=1,2,3.
\] 
We remark then that, for $i,j\in\{1,2,3\}$, $U^{ij}$ is homogeneous in $\RR^3\setminus\{0\}$ with degree $-1$. Moreover $U^{ij}$ satisfies that 
\beq
\Delta U^{ij}=\p_i q_j~~\mathrm{in}~~\RR^3\setminus\{0\},
\enq
where for each $j\in\{1,2,3\}$, $q_j$ is harmonic in $\RR^3\setminus\{0\}$. We can apply then Lemma \ref{modified goal} to the computation of the components of $A^{(n+1)}_{l}$ by choosing $V:=U^{ij}$ and $Q:=\p_i q_j$ for each $i,j\in\{1,2,3\}.$ This yields that, for $\varepsilon_0$ sufficiently small 
and arbitrary $q \in (1,\infty)$
\[
\Big(\sum_{l=1}^N |A^{(n+1)}_{l}|^q\Big)^{1/q}\leq C(q,\varepsilon_0)\left(\frac{a}{d}\right)^{3-3/q} \left( \sum_{l=1}^{N}|\mathbb M_l^{(n)}|^q \right)^{1/q}.
\]
However, by Proposition \ref{asym-stokes}, there exists an absolute constant
(independent of $n,l$ and other parameters) such that $|\mathbb M_l^{(n)}| \leq C|A^{(n)}_l|.$
This completes the proof of the proposition.
\end{proof}
%

We proceed with the analysis of the quality of the sequence of  approximations
$(u^{(n)}_{app})_{n\in\mathbb N}$.

\begin{proposition} \label{prop_perfore}
Let  $\varepsilon_0$ sufficiently small. There exists a constant $C_{app}(\varepsilon_0)$, such that for $n \geq 3$ and $\varepsilon < \varepsilon_0,$ there holds
\beq
\|u_N-u^{(n)}_{app}\|_{\dot{H}^{1}(\mathbb R^3)}\leq C_{app}(\varepsilon_0)|A|\big(\frac{a}{d}\big)^{11/2} .
\enq
\end{proposition}
\begin{proof}

By substracting the equations satisfied by 
$(u_N,p_N)$ and $(u^{(n)}_{app},p^{(n)}_{app}),$ we obtain that $\delta_u  = u_N - u^{(n)}_{app},$ $\delta_p = p_N - p^{(n)}_{app}$
satisfies:
\begin{align}
\label{eq_Stokes_delta u}
& \left\{
\begin{aligned}
- {\rm div}\, \Sigma_{\mu}(\delta_u,\delta_p) &= 0 \\
{\rm div}\, \delta_u &= 0 
\end{aligned}
\right.
\qquad \qquad
\text{ in $\mathbb R^3 \setminus \bigcup_{l=1}^N \overline{B}_l,$}
\end{align}
and
\begin{align}
\label{eq_Newton_deltaU}
\left\{
\begin{aligned}
\int_{\partial B_l} \Sigma_{\mu}(\delta_u,\delta_p) n {\rm d}s &= 0 \\
\int_{\partial B_l} (x-x_l) \times \Sigma_{\mu}(\delta_u,\delta_p) n {\rm d}s &= 0 \\
\end{aligned}
\right.
\qquad \qquad  \text{for }\, l = 1,\ldots,N.
\end{align}
As for boundary conditions, we note that by  definition of $u_{app}^{(n)}$, we have that
\begin{align*}
\delta_u(x)& =u_N(x)-Ax-\sum_{j=1}^{n} v^{(j)}(x), \qquad \text{ in $\mathcal F_N$} 
\end{align*}
Thanks to \eqref{eq_integrabilite} and extending the velocities $u_{app}^{(n)}$
and $u_N$ inside the particle domains with their boundary values,  
we have that $\delta_u \in \dot{H}^1(\mathbb R^3).$ On the boundaries,
reorganizing the terms involved in $v^{(j)},$ see also \eqref{eq_splitboundary},
we have that there exists vectors $(W_{l},\varpi_l)_{l=1,\ldots,N}$ for which:
\begin{equation} \label{bc_delta_u}
   \delta_u(x)=W_l+\varpi_l\times(x-x_l)+u_{l,n}^*, ~~\mathrm{on} ~~B_l,~~\forall l=1,\ldots,N,\\
\end{equation}
where we have $u_{l,n}^*=S^{(n)}_l+R^{(n)}_l$ with:
\beq
S^{(n)}_l(x):=\sum_{l\neq\lambda}a^3\mathcal{K}[A^{(n-1)}_\lambda,\mathcal{B}_\lambda](x-x_\lambda),
\enq
and
\beq
R^{(n)}_l(x)&:=&\sum_{j=0}^{n-2}\sum_{l\neq\lambda}a^3\mathcal{K}[A^{(j)}_\lambda,\mathcal{B}_\lambda](x-x_\lambda)-\sum_{j=0}^{n-2}\sum_{l\neq\lambda}a^3\mathcal{K}[A^{(j)}_\lambda,\mathcal{B}_\lambda](x_l-x_\lambda)\\
&&-\sum_{j=0}^{n-2}\sum_{l\neq\lambda}a^3[\nabla \mathcal{K}[A^{(j)}_\lambda,\mathcal{B}_\lambda](x_l-x_\lambda)](x-x_l)+\sum_{j=0}^{n-1}\sum_{l\neq\lambda}\mathcal{H}[A^{(j)}_\lambda,{B}_\lambda](x-x_\lambda).
\enq

We notice that for each $l=1,\dots,N$, the formula defining $u^*_{l,n}$ can be extended to $B(x_l,2a)$. We also mention that again
classical integration by parts arguments yield that $\delta_u$ realizes
\begin{equation}\label{min}
\min\Big\{ \int_{\mathcal{F}_N}|D(v)|^2, \; v\in\dot{H}^{1}(\mathbb R^3),\quad {\rm div}\ v=0, \;  D(v - u_{l,n}^*) = 0 \text{ on $B_l$},~\forall\, l\Big\}.
\end{equation}
The proof of our theorem then reduces to construct divergence-free vector-fields $w_{l,n} \in C^{\infty}_c(B(x_l,2a))$ that match $u^*_{l,n}$ 
(up to a rigid vector-field) on $B_l$ for each $l=1,\dots,N$. 
Indeed, since $\delta_u$ is divergence-free and using the minimizing principle
of \eqref{min}, we have then:
\[
\int_{\mathbb R^3} |\nabla \delta_u|^2 = 2 \int_{\mathbb R^3} |D(\delta_u)|^2 
 \leq  C \sum_{l=1}^N \int_{B(x_l,2a)} |\nabla w_{l,n}|^2.
\]
So, we define:
\beq
w_{l,n}(x):=\sum_{l=1}^N\big(\chi(\frac{x-x_l}{a})(u^*_{l,n}(x)-\bar{u}^*_{l,n})-\mathfrak{B}_{B(x_l,2a)\setminus\overline{ B(x_l,a)}}[(u^*_{l,n}(x)-\bar{u}^*_{l,n})]\cdot \nabla \chi(\frac{x-x_l}{a})\big),
\enq
Here, we denoted $\chi\in C_0^\infty(\RR^3)$ such that $\chi\equiv1$ on $B(0,3/2)$ and $\chi\equiv0$ in $\mathbb R^3 \setminus \overline{B(0,2)}$, $\bar{u}^*_{l,n}$ is the mean-value of $u^*_{l,n}$ over $B(x_l,2a)$. Clearly, our candidate matches the condition 
\[
w_{l,n}(x)=\mathrm{constant}+u^*_{l,n}, \quad \text {on $B_l.$}
\]
For the  next computations,  we introduce also $\bar{S}^{(n)}_l$ and $\bar{R}^{(n)}_l$  the mean-values of $S^{(n)}_l$ and $R^{(n)}_l$ over $B(x_l,2a)$ respectively so that $\bar{u}_{l,n}^* = \bar{S}^{(n)}_l + \bar{R}^{(n)}_l$.

\medskip

By the scaling properties of the Bogovskii operator, we obtain that 
\beq
\int_{\mathcal{F}_N}|\nabla w_{l,n}|^2& \leq & C\sum_{l=1}^N\|\nabla(\chi(\frac{\cdot-x_l}{a})(u^*_{l,n}-\bar{u}^*_{l,n}))\|^2_{L^2(B(x_l,2a))}\\
&\lesssim&\sum_{j=1}^4 H^{(n)}_{l,j},
\enq
where
\beq
H^{(n)}_{l,1}:=\sum_{l=1}^N\|\nabla S^{(n)}_l\|^2_{L^2(B(x_l,2a))},~H^{(n)}_{l,2}:=\sum_{l=1}^N\|\nabla R^{(n)}_l\|^2_{L^2(B(x_l,2a))},
\enq
and
\beq
H^{(n)}_{l,3}:=\dfrac{1}{a^2}\sum_{l=1}^N\| S^{(n)}_l-\bar{S}^{(n)}_l\|^2_{L^2(B(x_l,2a))},~~H^{(n)}_{l,4}:=\dfrac{1}{a^2}\sum_{l=1}^N\| R^{(n)}_l-\bar{R}^{(n)}_l\|^2_{L^2(B(x_l,2a))}.
\enq
Here, it is standard that the Poincar\'e-Wirtinger inequality entails that $H^{(n)}_{l,3}\leq CH^{(n)}_{l,1}$ and $H^{(n)}_{l,4}\leq CH^{(n)}_{l,2}$. Hence, we only need to bound $H^{(n)}_{l,1}$ and $H^{(n)}_{l,2}$.

\medskip

We deal with $H^{(n)}_{l,1}$ first. According to the definition of $S^{(n)}_l$
and $\mathcal K[A_{\lambda}^{(n-1)},\mathcal B_{\lambda}]$ (see Proposition \ref{asym-stokes}), for each $l=1,\dots N$, we have 
\beq
S^{(n)}_l&=&\sum_{l\neq\lambda}a^3\mathcal{K}[A^{(n-1)}_\lambda,\mathcal{B}_\lambda](x-x_\lambda)\\
&=&a^3\sum_{l\neq\lambda}(\mathbb{M}[A^{(n-1)}_\lambda,\mathcal{B}_\lambda]:\nabla \mathcal{U}^1(x),\mathbb{M}[A^{(n-1)}_\lambda,\mathcal{B}_\lambda]:\nabla \mathcal{U}^2(x),\mathbb{M}[A^{(n-1)}_\lambda,\mathcal{B}_\lambda]:\nabla \mathcal{U}^3(x)).
\enq
As in the proof of {Proposition \ref{Al}}, for each $i,j\in\{1,2,3\}$, $U^{ij}$ is homogeneous in $\RR^3\setminus\{0\}$ with degree $-1$ such that 
\beq
\Delta U^{ij}=\p_i q_j~~\mathrm{in}~~\RR^3\setminus\{0\},
\enq
where for each $j\in\{1,2,3\}$, $q_j$ is harmonic in $\RR^3\setminus\{0\}$. By the definition of $A^{(n+1)}_{l}$ and applying Lemma \ref{modified goal} by choosing $V:=U^{ij}$ and $Q:=\p_i q_j$ for each $i,j\in\{1,2,3\}$ and Proposition  \ref{Al}, we have
\[
H^{(n)}_{l,1}\leq [C(2,\varepsilon_0)]^{n-1} \big(\frac{a}{d}\big)^{3n}(a^3N)|A|^2.
\]
Up to restrict the size of $\varepsilon_0$ we obtain that:
\begin{equation} \label{eq_Hl1}
H^{(n)}_{l,1}\leq C(\varepsilon_0) \big(\frac{a}{d}\big)^{8}(a^3N)|A|^2.
\end{equation}

Now we turn to deal with $H^{(n)}_{l,2}$. By the definition of $R^{(n)}_l$, we have that for any $x\in B(x_l,2a)$,
\beq
|\nabla R^{(n)}_l(x)|\leq Ca^4\sum_{j=0}^{n-1}\sum_{l\neq\lambda}|A^{(j)}_\lambda|\frac{1}{|x_l-x_\lambda|^4}.
\enq

We notice here -- since the minimum distance between two $x_l$'s is lager than $d$ which is much larger than $a$ (for small $\varepsilon_0$) -- that, for each $l=1,\dots,N$  and for any $x\in B(x_l,2a)$, there holds:
\beq
\sum_{l\neq\lambda}|A^{(j)}_{\lambda}|\frac{1}{|x_l-x_\lambda|^4}&\leq& Cd^{-3}\sum_{l\neq \lambda}\int_{B(x_\lambda,d/2)}\dfrac{|A^{(j)}_\lambda|}{|x_l-y|^4}dy\\
& \leq&  Cd^{-3}\int_{\RR^3}|\Phi^{(j)}(y)|\frac{\mathbf{1}_{|x_l-y|>d/2}}{|x_l-y|^4}dy\\
&\leq&C(\varepsilon_0)d^{-3}\int_{\RR^3}|\Phi^{(j)}(y)|\frac{\mathbf{1}_{|x-y|>d/2}}{|x-y|^4}dy,
\enq
where 
\beq
\Phi^{(j)}(x):=\sum_{l=1}^NA^{(j)}_l\mathbf{1}_{B(x_l,d/2)}(x).
\enq
Therefore we obtain, with a direct Young inequality for convolution:
\beq
H^{(n)}_{l,2}\leq C\frac{a^8}{d^6}\|\sum_{j=0}^{n-1}|\Phi^{(j)}|*\frac{\mathbf{1}_{|y|>d/2}}{|y|^4}\|^2_{L^2(\RR^3)}\leq C\frac{a^8}{d^8}\Big(\sum_{j=0}^{n-1}\|\Phi^{(j)}\|_{L^2(\mathbb R^3)}\Big)^2,
\enq
which combined with Proposition \ref{Al}, yields that: 
\begin{equation} \label{eq_Hl2}
H^{(n)}_{l,2}\leq \Big(\frac{a}{d}\Big)^{8}\sum_{j=0}^{n-1}\Big( C(2,\varepsilon_0) \left(\frac{a}{d} \right)^{3/2}\Big)^{2j}Na^3|A|^2\leq C(\varepsilon_0)\Big(\frac{a}{d}\Big)^{8} Na^3 |A|^2,
\end{equation}
where we have chosen $\varepsilon_0$ sufficiently small so that the series
$(\sum_{j\geq 0}C(2,\varepsilon_0)(a/d)^{3/2})^{2j}$ converges.
Combining \eqref{eq_Hl1}-\eqref{eq_Hl2}, we obtain the expected result.
\end{proof}

\section{Approximation of the target system} \label{sec_continu}
In this section, we fix $A \in {\rm Sym}_{3,\sigma}(\mathbb R)$ and $\mathbb M_{eff} \in \mathcal M(\varepsilon_0)$ (for some small $\varepsilon_0$) and we 
 analyse the properties of the asymptotic problem
 \begin{align} \label{eq_continu_sec}
  & \left\{  
 \begin{aligned}
 - {\rm div} (2\mu(1+ \mathbb M_{eff})(D(u))  - p \mathbb{I}_3) &=& 0 \\
 {\rm div} u &=& 0 
 \end{aligned}
 \right. && \text{ on $\mathbb R^3$}, \\
  & \quad u(x) = A x && \text{ at infinity}.  \label{bc_continu_sec}
 \end{align}
 We note that $\mu$ is a gain a simple factor in this equation so that we only treat the case $\mu = 1$ below.
This system is associated with the weak formulation:
\begin{center}
\begin{minipage}{.8\textwidth}
Find $v \in \dot{H}^1_{\sigma}(\mathbb R^3)$ such that:
\[
2 \mu \int_{\mathbb R^3} [( 1 + \mathbb{M}_{eff} )(D(v) + A)] : \nabla w = 0\,, 
\quad \forall \, w \in \dot{H}^1_{\sigma}(\mathbb R^3).
\]
\end{minipage}
\end{center}
Since $\mathbb{M}_{eff}$ has compact support, for $\varepsilon_0$ sufficiently small we have that $\|\mathbb{M}_{eff}\|_{L^{\infty}(\mathbb R^3)} \leq 1/2$ so that 
construction of a weak solution falls into the scope of the Lax-Milgram theorem. Hence, under the assumption that $\varepsilon_0$ is sufficiently small we have existence and uniqueness of a $u_c$ satisfying
\begin{itemize}
\item $v_c(x) = (u_c(x) -Ax) \in \dot{H}^1_{\sigma}(\mathbb R^3),$ 
\item there exists $p_c$ for which \eqref{eq_continu_sec} holds true (in $\mathcal D'(\mathbb R^3)$),
\end{itemize}
and consequently, the pressure $p_c$ exists and is unique up to a constant.
We focus now --  as in the previous section for the problem in a perforated domain -- on  a possible expansion of the solution $u_c$ in terms of "powers of $\mathbb{M}_{eff}$". Namely, for small $\varepsilon_0$ the matrix   
$\mathbb{M}_{eff}$ can be seen as a perturbation of the identity so that
one may look for a solution to \eqref{eq_continu_sec}-\eqref{bc_continu_sec}
by iterating the mapping $v \mapsto \tilde{v} := \mathcal Lv$ solving the system:
\begin{align*}
 &\left\{ \begin{aligned}
 - {\rm div} (   D(\tilde{v})  - p \mathbb{I}_3) &=  {\rm div} \, \mathbb{M}_{eff}(D(v)) + {\rm div}  \, \mathbb{M}_{eff}(A) \\
 {\rm div} \tilde{v} &= 0 
 \end{aligned}
 \right.
 &&
 \text{ in $\mathbb R^3,$}\\
  & \quad \tilde{v}(x) = 0 && \text{ at infinity},
 \end{align*}
 starting form $v^{(0)}(x) = 0.$ Again, it is standard by introducing a weak formulation and Lax-Milgram arguments that there exists a unique velocity-field  $\tilde{v}_c$ satisfying  
\begin{itemize}
\item  $\tilde{v}_c \in \dot{H}^1_{\sigma}(\mathbb R^3)$
\item there exists a pressure $\tilde{p}_c$ such that  
\begin{equation*}
\left\{ \begin{aligned}
 - {\rm div} (   D(\tilde{v}_c)  - \tilde{p}_c \mathbb{I}_3) &=  {\rm div} \, \mathbb{M}_{eff}(A) \\
 {\rm div} \tilde{v} &= 0 
 \end{aligned}
 \right.
 \qquad
 \text{ in $\mathbb R^3,$}
 \end{equation*}
\end{itemize} 
The main result of this section is the following proposition which 
compares the velocity-field $v_c(x) = u_c(x) - Ax $ with $\tilde{v}_c.$ 

\begin{proposition}\label{uc-uc}
Under the assumption that $\varepsilon_0>0$ is sufficiently small, there exists a constant $C(K,\varepsilon_0)$ such that
\beqq\label{ucuc}
\|\nabla v_c-\nabla\tilde{v}_c\|_{L^2(\RR^3)}\leq C(K,\varepsilon_0)\|\mathbb{M}_{eff}\|^2_{L^{\infty}(\RR^3)} |A|.
\enqq
\end{proposition}
\begin{proof}
This proof is a straightforward application of fixed-point arguments. 
First, let prove that the mapping $\mathcal L$ is a contraction on $\dot{H}^1_{\sigma}(\mathbb R).$  Indeed, for arbitrary $(v_1,v_2) \in \dot{H}^1_{\sigma}(\mathbb R^3),$ given the weak formulation for the Stokes problem, we have that $w = \mathcal L(v_1-v_2)$ satisfies:
\[
\int_{\mathbb R^3} D(w) : D(\varphi) = \int_{\mathbb R^3} [\mathbb{M}_{eff} (D(v_1 -v_2))] : D(\varphi)\,, \qquad \forall \, \varphi \in \dot{H}^1_{\sigma}(\mathbb R^3). 
\]
Setting $w = \varphi$ and recalling that $w$ is divergence-free we obtain that 
\[
\int_{\mathbb R^3} |\nabla w|^2 = 2 \int_{\mathbb R^3} |D(w)|^2 
\leq 8 \|\mathbb{M}_{eff}\|^2_{L^{\infty}} \int_{\mathbb R^2} |D(v_1-v_2)|^2 
\leq 4\|\mathbb{M}_{eff}\|^2_{L^{\infty}} \int_{\mathbb R^2} |\nabla v_1-\nabla v_2|^2. 
\]
Consequently, fix $\varepsilon_0 < 1/8.$ Then $\|\mathcal L\| < 1/\sqrt{2}$ and the mapping $\mathcal L$ is a contraction that admits a unique fixed point. This yields a solution to \eqref{eq_continu_sec}-\eqref{bc_continu_sec}. Furthermore, this solution is obtained by iterating the mapping $\mathcal L$ from  $v^{(0)} = 0.$
So the sequence $v^{(n)} = \mathcal L^{\circ n} v^{(0)}$ converges to $v_c$
(in $\dot{H}^{1}(\mathbb R^3)$) while, by definition, $\tilde{v}_c = v^{(1)}.$ 
Similar energy estimates yield that 
\begin{equation} \label{eq_boundvtilde}
\|\nabla \tilde{v}_c\|_{L^2(\mathbb R^3)}
\leq |K|^{1/2} \|M_{eff}\|_{L^{\infty}(\mathbb R^3)}|A|.
\end{equation}
Standard arguments with contractions then yield that
\begin{align*}
\|\nabla \tilde{v}_c - \nabla v_c  \|_{L^2(\mathbb R^3)} 
& = \|v^{(1)} - \lim v^{(n)}\|_{\dot{H}^1(\mathbb R^3)}  \\
& \leq  \dfrac{\|\mathcal L \|}{1 - \| \mathcal L\|} \|v^{(1)}\|_{\dot{H}^1(\mathbb R^3)} \leq  4|K|^{1/2} \|\mathbb{M}_{eff}\|_{L^{\infty}(\mathbb R^3)}  \|M_{eff}\|_{L^{\infty}(\mathbb R^3)} |A|.
\end{align*}
 This concludes the proof.
\end{proof}

\section{Proof of main result}
\label{sec_mainresult}
We end the paper with a proof of  Theorem \ref{thm_main}. In the whole section,  we assume that we are given a perforated domain such that  \eqref{H1}-\eqref{H2} hold true. We are also given $\mathbb{M}_{eff} \in \mathcal M(\varepsilon_0)$ with simultaneously $(a/d)^3 < \varepsilon_0$
(see \eqref{H1}-\eqref{H2} for the definitions of $a$ and $d$). 
Restrictions on $\varepsilon_0$ are introduced throughout the section. 
Finally, we fix a matrix $A \in {\rm Sym}_{3,\sigma}(\mathbb R).$

\medskip

We recall that Theorem \ref{thm_main} is a stability estimate between the solutions to two problems. The first one is the Stokes problem in a perforated domain \eqref{eq_Stokes}-\eqref{bc_Stokes}-\eqref{eq_Newton}
that we studied in Section \ref{sec_refmet}. We restrict at first $\varepsilon_0$ so that Proposition \ref{prop_perfore}  holds true. We have then a sequence of approximations $(u_{app}^{(n)})_{n\in \mathbb N}$ to the velocity-field $u_N(x) = Ax + v_N(x)$ solution to \eqref{eq_Stokes}-\eqref{bc_Stokes}-\eqref{eq_Newton}. The second problem is the continuous analogue
\eqref{eq_continu}-\eqref{bc_continu} that we studied in Section \ref{sec_continu}. We assume also that $\varepsilon_0$ is sufficiently small so that Proposition \ref{uc-uc} holds true. We have then an approximation
$\tilde{u}_{c}(x) = Ax + \tilde{v}_c(x)$ to the velocity-field $u_c$ solution to 
\eqref{eq_continu}-\eqref{bc_continu}.

The purpose of Theorem \ref{thm_main} is to compute a bound from above for 
$u_N - u_c.$ To this end, we fix $n=3$ and $u_{app} = u^{(3)}_{app}$ (with the notations of Proposition \ref{prop_perfore})  and write 
\begin{equation} \label{eq_errorsplitting}
u_N - u_c =  (u_N - u_{app})  + (u_{app} - \tilde{u}_c) + (\tilde{u}_{c} - u_c ) 
=: R_{perf} +   R_{main} + R_{cont}.
\end{equation}

The two error terms $R_{perf}$ and $R_{cont}$ have been estimated previously
in Proposition \ref{prop_perfore} and Proposition \ref{uc-uc} respectively.  So, we proceed in the next subsection with estimating $R_{main}$
and shall combine the various partial results in the last subsection to complete the proof of our theorem.

\subsection{Computing $R_{main}$}
We recall that $u_{app} = u^{(3)}_{app}$ is constructed via the method of reflections: 
\beq
u_{app}(x)=Ax+\sum_{j=1}^3\Big(\sum_{l=1}^N U[A^{(j-1)}_l,B_l](x-x_l)\Big).
\enq
By the definition of $u_{app}$ and Proposition \ref{asym-stokes}, we have the following decomposition, for any $x\in\mathcal{F}_N$
\beqq
u_{app}(x) - \tilde{u}_c(x)=R_1(x)+R_2(x),
\enqq
where 
\beq
R_1(x):=\sum_{l=1}^N\mathcal{K}[A,B_l](x-x_l) - \tilde{v}_c(x)
\enq
and
\beq
R_2(x):=\sum_{l=1}^N\mathcal{H}[A,{B}_l](x-x_l)+\sum_{j=2,3}\sum_{l=1}^N U[A^{(j-1)}_l,B_l](x-x_l).
\enq
We start with computing $R_1:$
\begin{proposition}\label{Homo-vc}
Let $K_0 \Subset \mathbb R^3$ and $p \in [1,3/2[,$ there exists $C(K_0)
$ for which:
\begin{equation} \label{eq_boundR1}
\|R_1\|_{L^{p}(K_0 \setminus \bigcup B(x_l,4a))}\leq C(K_0)\Big[\|\mathbb M_N -\mathbb M_{eff}\|_{\dot{H}^{-1}(\RR^3)} + \left( \dfrac{a^3}{d^3}\right)^{1+\theta} \Big],
\end{equation}
where $\theta = \frac1p - \frac 23.$
\end{proposition}
\begin{proof}
By Proposition \ref{asym-stokes}, we know that each component of $\mathcal{K}$ can be written as:
\beq
\mathcal{K}[A,B_l]_i(x-x_l)&=&\mathbb{M}[A,B_l]:\nabla\mathcal{U}^i(x-x_l)
\enq
In this identity, we recall that $\mathbb M[A,B_l] = a^3 \mathbb M[A,\mathcal B_l],$ 
given by \eqref{eq_defM}, and that $\mathcal{U}^i=(U^{i1},U^{i2},U^{i3})$  with
\beq
U^{ij}:=-\frac{1}{8\pi}\big[\frac{\delta_{ij}}{|x-y|}+\frac{(x_i-y_i)(x_j-y_j)}{|x-y|^3}\big]
\enq
corresponding to the fundamental solution of Stokes equation in $\RR^3$. 
According to the fact that for any $x\in\RR^3\setminus\{0\}$
\[
\Delta\mathcal{U}^i(x)=\nabla q_i
\]
with $q_i(z)=\frac{1}{4\pi}\frac{z}{|z|^3}$ and applying Lemma \ref{mean-value}, we obtain that for any $x \in K_0 \setminus \bigcup B(x_{\lambda},a)$ and $i=1,2,3$, for any $l=1,\ldots,N:$
\begin{multline*}
\mathcal{K}[A,{B}_l]_i(x-x_l)= \frac{3}{4\pi} \mathbb{M}[A,\mathcal{B}_l]:  \int_{B(x_l,a)}\nabla \mathcal{U}^i(x-y)dy \\
+ \dfrac{a^3}{3} \mathbb{M}[A,\mathcal{B}_l]: \int_0^a\left(\frac{r^4}{a^3}-r\right)\fint_{B(x_l,r)}\nabla^2 q_i(x-y)dydr
\end{multline*}
We proceed by remarking that:
\beqq
\sum_{l=1}^N \dfrac{3}{4\pi}\mathbb{M}[A,\mathcal{B}_l]: \int_{B(x_l,a)}\nabla \mathcal{U}^i(x-y)dy=\int_{\RR^3} \mathbb M_N(A)(y) :\nabla\mathcal{U}^i(x-y)dy
\enqq
Furthermore, since $q_i$ is harmonic on $\mathbb R^3 \setminus \{0\},$
for $x \in K_0 \setminus \bigcup B(x_{\lambda},a)$ and $l\in \{1,\ldots,N\}$ there holds:
\beqq
\begin{split}
\int_0^a\left(\frac{r^4}{a^3}-r\right)\fint_{B(x_l,r)}\nabla^2 q_i(x-y)dydr&=\int_0^a\left(\frac{r^4}{a^3}-r\right)\fint_{B(x_l,a)}\nabla^2 q_i(x-y)dydr\\
&= - \frac{9}{40\pi a}\int_{B(x_l,a)}\nabla^2 q_i(x-y)dy.
\end{split}
\enqq

On the other hand, by uniqueness of the solution to the Stokes problem 
defining $\tilde{v}_c$ (in $\dot{H}^1(\mathbb R^3)$), we know that 
$\tilde{v}_c$ is computed  with Green's function for the Stokes problem.
This yield componentwise:
\[
\tilde{v}_{c,i}(x) = \int_{\mathbb R^3} \mathbb{M}_{eff}(A)(y) :  \nabla \mathcal U^{i}(x-y){\rm d}y, \quad \forall \, x \in \mathbb R^3.
\]
We note that this quantity is well-defined since $\mathbb M_{eff} \in L^{\infty}(\mathbb R^3)$ has compact support and $\nabla \mathcal U^i$
is homogeneous of degree $-2$ so that it is $L^p_{loc}(\mathbb R^3)$
for arbitrary $p < 3/2.$
Eventually, the $i-$th component of  $R_1$ can be rewritten as
\begin{multline*}
R_{1,i}(x)= \int_{\RR^3} \left[ \mathbb M_N - \mathbb M_{eff} \right](A)(y):  \nabla\mathcal{U}^i(x-y)dy \\
 - \frac{3a^2}{40\pi}\int_{\mathbb R^3} \mathbb M_N(A)(y) : \mathbf{1}_{|x-y| > 3a}\nabla^2 q_i(x-y)dy,
\end{multline*}
since $x \notin \bigcup B(x_l,4a).$
Concerning the first term on the right-hand side of this equality,
let denote $h$ any component of $\left[\mathbb M_N - \mathbb M_{eff} \right].$
By assumption, we have then $h\in \dot{H}^{-1}(\RR^3)$ so that it can be written $h=\partial_1 \varphi_1+\partial_2 \varphi_2$, where  $\varphi_1$ and $\varphi$ are  $L^2(\RR^3)$ and 
\[
\max_{i=1,2}\|\varphi_i\|_{L^2(\RR^3)}\leq \|h\|_{\dot{H}^{-1}(\RR^3)} \leq 
\|\mathbb M_N - \mathbb M_{eff} \|_{\dot{H}^{-1}(\RR^3)}.
\]
Therefore, we have on $\mathbb R^3:$
\[
 \int_{\RR^3} h (y) \partial_{l}{U}^{ik}(x-y)dy = \int_{\RR^3} \varphi_1(y)   \partial_{1l}{U}^{ik}(x-y)dy + \int_{\RR^3} \varphi_2(y)  \partial_{2l}{U}^{ik}(x-y)dy,
\]
where, by Calder\'on Zygmund inequality, the right-hand side of this identity is well defined and satisfies.
\begin{equation} \label{eq_R11}
\|\int_{\RR^3} h  \partial_{l}{U}^{ik}(x-y)dy \|_{L^2(\RR^3)}
\leq 
C\|\mathbb M_N - \mathbb M_{eff} \|_{\dot{H}^{-1}(\RR^3)}.
\end{equation}

As for the second term in $R_{1,i}$ we can apply a classical Young inequality for convolution to obtain:
\begin{align*}
\|\int_{\RR^3} \mathbb{M}_N(A)(y)\nabla^2 q_i(x-y)dy\|_{L^{p}(\mathbb R^3)}
& \leq C\|\mathbb{M}_N(A)\|_{L^{p}(\RR^3)}\|\dfrac{\mathbf{1}_{|z|>3a}}{|z|^4}\|_{L^1(\mathbb R^3)} \\
& \leq C|A|\Big(\frac{a}{d}\Big)^{\frac 3p} \dfrac{1}{a}.
\end{align*}
Finally, applying the embedding $L^{2}(K_0) \subset L^{p}(K_0),$ we have:
\[
\|R_1\|_{L^p(K_0 \subset \bigcup B(x_l,4a))} \leq C|A| \left[ \|\mathbb M_N(y) - \mathbb M_{eff}(y) \|_{\dot{H}^{-1}(\RR^3)} + a \Big(\frac{a}{d}\Big)^{\frac 3p}. \right]
\]
We conclude by applying again that $Nd^3 \leq |K|$ so that 
\[
a \left( \dfrac{a}{d}\right)^{\frac 3p} \leq (N a^3)^{\frac 13} \left( \dfrac{a}{d}\right)^{\frac 3p} 
\leq \left[\left( \dfrac{a}{d}\right)^3\right]^{\frac 1p+\frac 13}.
\]
\end{proof}

We proceed by computing error estimates for $R_2.$
This is the content of the following proposition: 
\begin{proposition} 
Let $K_0 \Subset \mathbb R^3$ and $p \in [1,3/2[,$ there exists $C(K,K_0)
$ for which:
\begin{equation} \label{eq_boundR2}
\|R_2\|_{L^p(K_0 \setminus \bigcup B(x_l,4a))} \leq C(K,K_0) \left( \dfrac{a^3}{d^3} \right)^{1+\theta} ,
\end{equation}
where $\theta = \frac 1p - \frac 23.$  
\end{proposition}
\begin{proof}
We recall that 
\beq
R_2(x)&=&\sum_{l=1}^N\mathcal{H}[A,{B}_l](x-x_l)+\sum_{j=2,3}\sum_{l=1}^N U[A^{(j-1)}_l,{B}_l](x-x_l)\\
&=&\sum_{j=2,3}\sum_{l=1}^N\mathcal{K}[A^{(j-1)}_l,{B}_l](x-x_l)+\sum_{j=1}^3\sum_{l=1}^N\mathcal{H}[A^{(j-1)}_l,{B}_l](x-x_l).
\enq
Since $B_l = a \mathcal B_l,$ by Proposition  \ref{asym-stokes}, we know that for any $j=1,2,3$ and $l=1,\dots,N$, when $|x-x_l| > 4a:$
\beq
|\mathcal{K}[A^{(j-1)}_l,{B}_l](x-x_l)|\leq C \frac{a^3|A^{(j)}_l|}{|x-x_l|^2},
\qquad
|\mathcal{H}[A^{(j-1)}_l,{B}_l](x-x_l)|\leq C \frac{ a^4|A^{(j)}_l|}{|x-x_l|^3}.
\enq
Therefore, on $K_0 \setminus \bigcup B(x_l,4a)$, we have:
\beq
|R_2(x)|\leq C \Big[a^3\sum_{l=1}^N\frac{|A^{(1)}_l|+|A^{(2)}_l|}{|x-x_l|^2}+a^4 \sum_{l=1}^N\frac{|A|}{|x-x_l|^3}\Big].
\enq
Consequently, denoting ${K_1} = K_0 \cup K,$ we obtain 
\begin{align*}
\|R_2\|_{L^p(K_0)}&  \leq C a^3 \sum_{i=1}^2\sum_{l=1}^{N} |A_l^{(i)}|\left( \int_{a}^{diam(K_1)} \dfrac{dx}{|x|^{2p}}   \right)^{\frac 1p}+ 
Na^4 |A| \left( \int_{a}^{\infty} \dfrac{dx}{|x|^{3p}} \right)^{\frac 1p} \\
& \leq C(K_1) \left( a^{1 + \frac 3p}\sum_{i=1}^{2} \sum_{l=1}^{N} |A_l^{(i)}| + 
Na^{1+\frac 3p} |A| \right) \\
& \leq C(K,K_1) \left(\dfrac{a^3}{d^3}\right)^{1+\theta}|A|
\end{align*}
where we applied Proposition \ref{prop_perfore} to pass from the second to the last line together with the remark that $Nd^{3} \leq |K|.$
\end{proof}

\subsection{End of the proof}

Let $K_0 \Subset \mathbb R^3$ containing $K$ (for simplicity) and 
$p\in[1,3/2[$.  By \eqref{eq_errorsplitting} we have
\[
\|u_N - u_c\|_{L^p(K_0)} \leq \|R_{perf}\|_{L^p(K_0)} + \|R_{main}\|_{L^p(K_0)}
 + \|R_{cont}\|_{L^p(K_0)}.
\]  
Since $p\leq 6$ and by the embedding $\dot{H}^1(\mathbb R^3) \subset L^{p}_{loc}(\mathbb R^3)$
and Proposition \ref{prop_perfore}, we have the bounds
\begin{align*}
 \|R_{perf}\|_{L^p(K_0)}& \leq C(K_0) \|R_{perf}\|_{L^6(\mathbb R^3)} \leq C(K_0)
 \|u_N - u_{app}\|_{\dot{H}^1(\mathbb R^3)}\\
& \leq C(\varepsilon_0,K_0)|A| \left(  \dfrac{a^3}{d^3} \right)^{\frac{11}{6}}.
\end{align*}
With a similar chain of inequalities, we obtain by applying Proposition \ref{uc-uc}
\begin{align*}
\|R_{cont}\|_{L^p(K_0)} \leq C(K,K_0,\varepsilon_0)|A| \|\mathbb{M}_{eff}\|^2_{L^{\infty}(\mathbb R^2)}.
\end{align*}
Finally, concerning $R_{main}$, we recall that we have simultaneously
$R_{main} = u_{app} - \tilde{u}_c = v_{app} - \tilde{v}_c$ (where we denote $v_{app}(x) = u_{app}(x) - Ax$) and $R_{main} = R_1 +R_2,$ 
with the notations of the previous subsection. This entails that:
\[
\|R_{main}\|_{L^p(K_0)} \leq 
 \|u_{app} - \tilde{u}_c\|_{L^p( \bigcup B(x_l,4a))}
+ \|R_1\|_{L^p(K_0\setminus \bigcup B(x_l,4a))} + \|R_{2}\|_{L^p(K_0\setminus \bigcup B(x_l,4a))}.
\]
The two last terms of the right-hand side are controlled respectively by 
\eqref{eq_boundR1} and \eqref{eq_boundR2}: 
\begin{multline} \label{eq_boundglobal}
\|R_1\|_{L^p(K_0\setminus \bigcup B(x_l,4a))} + \|R_{2}\|_{L^p(K_0\setminus \bigcup B(x_l,4a))}
\\ \leq  C(K,K_0)|A| \left[ \|\mathbb M_N -\mathbb M_{eff}\|_{\dot{H}^{-1}(\RR^3)} + \left( \dfrac{a^3}{d^3}\right)^{1+\theta}  \right]
\end{multline}
where $\theta = \frac 1p-\frac 23.$
As for the first term, we  first bound
\[
\|u_{app} - \tilde{u}_c\|_{L^p(\cup B(x_l,4a))} \leq |\bigcup B(x_l,4a)|^{\frac 1p - \frac 16}( \|v_{app}\|_{L^6(\mathbb R^3)} +  \|\tilde{v}_c\|_{L^6(\mathbb R^3)}). 
\]
Here, it is straightforward from \eqref{eq_boundvtilde} that:
\[
\|\tilde{v}_c\|_{L^6(\mathbb R^3)} \leq C(K)|A| \|\mathbb M_{eff}\|_{L^{\infty}(\mathbb R^3)}.
\]
As for $u_{app},$ we have, by Proposition \ref{prop_perfore} and 
uniform estimate \eqref{eq_unifuN} that:
\begin{align*}
\|v_{app}\|_{L^6(\mathbb R^3)}&  \leq C \|\nabla v_{app}\|_{L^6(\mathbb R^3)} \leq  C \left( \|\nabla v_{app} - \nabla v_N\|_{L^2(\mathbb R^3)} + \|\nabla v_N\|_{L^2(\mathbb R^3)}\right)\\
& \leq C|A| \left(\dfrac{a^3}{d^3} \right)^{\frac 12}.
\end{align*}
Via a straightforward bound on the volume of the $B(x_l,4a)$ we conclude that:
\begin{equation} \label{eq_boundlocal}
\|u_{app} - \tilde{u}_c\|_{L^p(\bigcup B(x_l,4a))}  \leq 
\left( \dfrac{a^3}{d^3}\right)^{\frac 1p - \frac 16}\left(\left(\dfrac{a^3}{d^3}\right)^{1/2} + \|\mathbb{M}_{eff}\|_{L^{\infty}(\mathbb R^3)} \right) |A|.
\end{equation}
Combining \eqref{eq_boundglobal} and \eqref{eq_boundlocal} yields
\begin{multline} 
\|R_{main}\|_{L^p(K_0)} \\
\leq  C(K,K_0,\varepsilon_0)|A| \left[ \|\mathbb M_N -\mathbb M_{eff}\|_{\dot{H}^{-1}(\RR^3)} + \left( \dfrac{a^3}{d^3}\right)^{1+\theta}   + \|\mathbb{M}_{eff}\|^2_{L^{\infty}(\mathbb R^3)} \right],
\end{multline}
since, as $p<3/2,$ we have $2/p-1/3 > 1 + \theta = 1/p+1/3.$

\medskip

Finally, we have proven:
\begin{multline*}
\|u_N - u_c\|_{L^p(K_0)}\\
 \leq C(K,K_0,\varepsilon_0) |A|\left[ \|\mathbb M_N[A] -\mathbb M_{eff}[A]\|_{\dot{H}^{-1}(\RR^3)} + \left( \dfrac{a^3}{d^3}\right)^{1+\theta}   + \|\mathbb{M}_{eff}\|^2_{L^{\infty}(\mathbb R^3)} \right].
\end{multline*}
This concludes the proof.

\bigskip

\paragraph{\bf Acknowledgement.}
The authors would like to thank David G\'erard-Varet  for many fruitful discussions. The authors are partially supported by ANR Project IFSMACS
ANR-15-CE40-0010. The first author is also supported by ANR Project
SingFlow ANR-18-CE40-0027 and Labex Numev Convention grants ANR-10-LABX-20.

\appendix

\section{Tools for the method of reflections} \label{app_ref}

In this appendix, we give some technical tools that are involved in the method of reflections.  We start with a representation formula generalizing the mean-value formula for harmonic functions.

\begin{lemma}\label{mean-value}
Suppose that $f\in L^1_{loc}(\RR^3)$. Let $D$ be a domain in $\RR^3$ and $\La u=f$ in $D$. Then for arbitrary $x \in D$ and $r >0$ such that $B(x,r) \subset D $ we have:
\beq
u(x)=\fint_{B(x,r)}u(y)dy+ \dfrac{1}{3} \int_0^r\left(\frac{\rho^4}{r^3} - \rho\right)\fint_{B(x,\rho)}f(y)dyd\rho.
\enq
\end{lemma}

\begin{proof}
This lemma must be part of the folklore. We give a proof for completeness.
Let
\[
\phi(r) = \fint_{\partial B(x,r)} u(y)d\sigma(y) = \dfrac{1}{4\pi}\int_{\partial B(0,1)} u(x+z) d\sigma(z).
\]
After differentiation and integration by parts, we obtain:
\[
\phi'(r) = \dfrac{1}{4\pi r^2}\int_{\partial B(x,r)} \partial_n u(y) d\sigma(y)
 = \dfrac{r}{3} \fint_{B(x,r)} f(y) d y.
\]
Since $\phi(r) \to 0$ when $r \to 0$ we infer that:
\[
\phi(r) = u(x) + \int_0^{r} \dfrac{\rho}{3} \fint_{B(x,\rho)} f(y) dy d\rho, \quad 
\text{ if $B(x,r) \subset D.$}
\]

\medskip

Then, if $B(x,r) \subset D$ we have:
\[
\fint_{B(x,r)} u(y) dy= \dfrac{1}{r^3} \int_{0}^r 3\rho^2 \phi(\rho) {d}\rho. 
\]
Integrating by parts and applying the formula we derived above for $\phi(r)$
and $\phi'(r)$, we obtain: 
\begin{align*}
\fint_{B(x,y)} u(y) dy &= \dfrac{1}{r^3} \left[ r^3 \phi(r) - \int_0^{r} \rho^3 \phi'(\rho) d\rho \right]\\
			& = \left(u(x) +  \int_0^{r} \dfrac{\rho}{3} \fint_{B(x,\rho)} f(y) dy d\rho \right) - \dfrac{1}{r^3} \int_0^r \dfrac{\rho^4}{3} \fint_{B(x,\rho)} f(y) dy d\rho.
\end{align*}
This concludes the proof.
\end{proof}

Relying on this formula, we analyze the behavior of the recursive formula for
the method of reflections \eqref{eq_iter_matrix}. We recall that we consider here a set of centers of mass $(x_1,\ldots,x_N)$ and parameters $a,\varepsilon_0$ such that 
$a^3/d^3 < \varepsilon_0$ where $d= \min_{i\neq j} |x_i-x_j|.$ We include
the recursive formula in the following more general framework. We assume we
are given $V \in C^{\infty}(\mathbb R^3 \setminus \{0\})$ homogeneous of degree $-1$ and we suppose that $\Delta V = Q$  where $Q$ is harmonic in $\RR^3\setminus\{0\}$ and homogeneous with degree $-3$. We look then 
at quantities of the form
\begin{equation} \label{eq_Walpha}
W_{l,\alpha}:=a^3\sum_{l\neq\lambda}m_{\lambda}\p^{\alpha} V(x_l-x_\lambda)
\quad 
\forall \, l =1,\ldots,N.
\end{equation}
where $(m_1,\ldots,m_N)$ are given and arbitrary and $\alpha$ is a multi-index in $\mathbb N^3.$  The crucial result underlying the method of reflections is the following lemma:
\begin{lemma}\label{modified goal}
Let $\varepsilon_0$ small and $1<q<\infty.$  Then,
there exists a constant $C(q,\varepsilon_0)$ such that  the following properties hold true:
\begin{enumerate}
\item if $|\alpha|=2$,  then
\beq
\Big(\sum_{l=1}^N|W_{l,\alpha}|^q\Big)^{1/q}\leq C(q,\varepsilon_0\Big)\big(\frac{a}{d}\big)^{3-3/q}\Big(\sum_{l=1}^N|m_l|^q\Big)^{1/q},
\enq

\item if $|\alpha|=1$ and we denote $F_{l}(x):=\sum_{l\neq\lambda}m_{\lambda}\nabla V(x-x_\lambda)$ on $B(x_l,2a)$, there holds:
\beqq\label{Fl-H^1}
\sum_{l=1}^N\|\nabla F_{l}\|^2_{L^2(B(x_l,2a))}\leq \dfrac{C(2,\varepsilon_0)}{d^{3}}\sum_{l=1}^N|m_l|^2.
\enqq

\end{enumerate}
\end{lemma}

\begin{proof} We split the proof into two parts corresponding to the two items
in the lemma.\\

\textbf{Part 1.}
This part is the proof of the first statement of the above lemma.
By definition of $Q,$ we notice that $\Delta\p^\alpha V=\p^\alpha Q$ in $\RR^3\setminus\{0\}$. Hence according to Lemma \ref{mean-value} and the fact that $\p^{\alpha}Q$ is harmonic in $\RR^3\setminus\{0\}$, we obtain that 
\beq
W_{l,\alpha}=W^{1}_{l,\alpha}+W^2_{l,\alpha}+W^{3}_{l,\alpha},
\enq
where
\beq
W^{1}_{l,\alpha}:=a^{3}\fint_{B(x_l,a)}\big(\sum_{l\neq \lambda}\fint_{B(x_\lambda,d/2)}m_{\lambda}\p^{\alpha}V(y-z)dz\big)dy,
\enq
\beq
W^{2}_{l,\alpha}:=\dfrac{a^3}{3}\int_0^a\left(\frac{r^4}{a^3}-r\right)\big(\sum_{l\neq\lambda}\fint_{B(x_l,r)}m_{\lambda}\p^\alpha Q(y-x_\lambda)dy\big)dr
\enq
and
\beq
W^{3}_{l,\alpha}:=\dfrac{a^3}{3}\fint_{B(x_l,a)}\big(\sum_{l\neq\lambda}\int_0^{\frac{d}{2}}\left(\frac{8r^4}{d^3}-r\right)\fint_{B(x_\lambda,r)}m_{\lambda}\p^\alpha Q(y-z)dzdr\big)dy.
\enq
For the next computations, we introduce:
\beqq\label{Phi}
\Phi(x):=\sum_{l=1}^N m_{l}\mathbf{1}_{B(x_l,d/2)}(x)
\enqq
and fix $q \in (1,\infty).$

\textbf{Step 1.}  In this part we deal with $W^1_{l,\alpha}$. By definition, we have that 
\beq
W^1_{l,\alpha}=\frac{9}{2\pi^2 d^3}\int_{B(x_l,a)}\int_{\RR^3\setminus B(x_l,d/2)}\Phi(z)\p^{\alpha}V(y-z)dzdy.
\enq
In order to apply  Calder\'on-Zygmund  inequality, we split the above quantity into two parts: $W^1_{l,\alpha}=C_{l,\alpha}+R_{l,\alpha},$
where
\begin{align*}
C_{l,\alpha}& :=\frac{9}{2\pi^2 d^3}\int_{B(x_l,a)}\int_{\RR^3\setminus B(y,d/2)}\Phi(z)\p^{\alpha}V(y-z)dzdy, \\
R_{l,\alpha}& :=\frac{9}{2\pi^2 d^3}\int_{B(x_l,a)}\Big(\int_{\RR^3\setminus B(x_l,d/2)}-\int_{\RR^3\setminus B(y,d/2)}\Big)\Phi(z)\p^{\alpha}V(y-z)dzdy.
\end{align*}
We note that for any $y\in \RR^3$,
\beq
\int_{\RR^3\setminus B(y,d/2)}\Phi(z)\p^{\alpha}V(y-z)dz=\int_{\RR^3\setminus B(0,d/2)}\Phi(y-z)\p^\alpha V(z)dz:=G_{\alpha}(y),
\enq
which implies that 
\beq
C_{l,\alpha}=\frac{9}{2\pi^2 d^3}\int_{B(x_l,a)}G_\alpha(y)dy.
\enq
Therefore, we have:
\beq
\big|C_{l,\alpha}\big|\leq C a^{3-3/q}d^{-3}\|G_\alpha\|_{L^{q}(B(x_l,a))},
\enq
and
\beq
\Big(\sum_{l=1}^N \big|C_{l,\alpha}\big|^q\Big)^{1/q}\leq Ca^{3-3/q}d^{-3}\|G_\alpha\|_{L^{q}(\RR^3)}.
\enq
On the other hand, by Calder\'on-Zygmund inequality we have:
\beq
\|G_{\alpha}\|_{L^q(\RR^3)}\leq C(q)\|\Phi\|_{L^q(\RR^3)}\leq C(q)d^{\frac{3}{q}}\Big(\sum_{l=1}^N |m_l|^q\Big)^{1/q}.
\enq
Hence we obtain that:
\beqq\label{Cl}
\Big(\sum_{l=1}^N\big|C_{l,\alpha}\big|^q\Big)^{1/q}\leq C(q)\big(\frac{a}{d}\big)^{3-3/q}\Big(\sum_{l=1}^N |m_l|^q\Big)^{1/q}.
\enqq

Now we turn to deal with $R_{l,\alpha}$. At first, we notice that for any $l=1,\dots,N$ and $x\in B(x_l,a)$, there holds:
\beq
B(x,d/2)\triangle B(x_l,d/2)\subset B(x,d/2+a)\setminus B(x,d/2-a),
\enq
(where $\triangle$ represents the symmetric difference between sets).
Since $\partial^{\alpha} V$ is $-3$-homogeneous, this implies that for any $y\in B(x_l,a)$, we have:
\beq
\Big|\Big(\int_{\RR^3\setminus B(x_l,d/2)}-\int_{\RR^3\setminus B(y,d/2)}\Big)\Phi(z)\p^{\alpha}V(y-z)dz\Big|
\leq \int_{B(y,d+2a)\setminus B(y,d-2a)}\big|\Phi(z)\frac{1}{|y-z|^3}\big|dz.
\enq
We denote $\bar{G}_\alpha(y)$ the right-hand side of this inequality.
Again by H\"older inequality, we obtain that 
\beq
\big|R_{l,\alpha}\big|\leq C a^{3-3/q}d^{-3}\|\bar{G}_\alpha\|_{L^q(B(x_l,a))},
\enq
and
\beq
\Big(\sum_{l=1}^N\big|R_{l,\alpha}\big|^q\Big)^{1/q}\leq C a^{3-3/q}d^{-3}\|\bar{G}_\alpha\|_{L^q(\RR^3)}.
\enq
On the other hand, by a standard Young inequality for convolution, we have
\beq
\|\bar{G}_{\alpha}\|_{L^q(\RR^3)}&=&\Big\|\int_{\RR^3}|\Phi(\cdot-z)|\frac{\mathbf{1}_{B(0,d+2a)\setminus B(0,d-2a)}(z)}{|z|^3}dz\Big\|_{L^q(\RR^3)}\\
&\leq&C\ln\big(\frac{d+2a}{d-2a}\big)\|\Phi\|_{L^q(\RR^3)}\leq C\ln\big(\frac{d+2a}{d-2a}\big)d^{3/q}\Big(\sum_{l=1}^N|m_l|^q\Big)^{1/q}.
\enq
Therefore we obtain that 
\beqq\label{Rl}
\Big(\sum_{l=1}^N\big|R_{l,\alpha}\big|^q\Big)^{1/q}\leq C\ln\big(\frac{d+2a}{d-2a}\big)\big(\frac{a}{d}\big)^{3-3/q}\Big(\sum_{l=1}^N|m_l|^q\Big)^{1/q}.
\enqq
By combing \reff{Cl} and \reff{Rl}, we obtain finally that,
\beqq\label{W1l}
\Big(\sum_{l=1}^N\big|W^1_{l,\alpha}\big|^q\Big)^{1/q}\leq C(q)\Big(1+\ln\big(\frac{d+2a}{d-2a}\big)\Big)\big(\frac{a}{d}\big)^{3-3/q}\Big(\sum_{l=1}^N|m_l|^q\Big)^{1/q}.
\enqq

\textbf{Step 2.} Now we turn to handle $W^2_{l,\alpha}$. We recall that  
\beq
W^{2}_{l,\alpha}=\dfrac{a^3}{3}\int_0^a\left(\frac{r^4}{a^3}-r\right)\big(\sum_{l\neq\lambda}\fint_{B(x_l,r)}m_{\lambda}\p^\alpha Q(y-x_\lambda)dy\big)dr.
\enq
Since $\p^{\alpha}Q$ is harmonic outside $B(0,a)$, we have for each $\lambda\neq l$ and $r<a$,
\beq
\fint_{B(x_l,r)}m_{\lambda}\p^{\alpha}Q(y-x_\lambda)dy=\fint_{B(x_l,a)}\fint_{B(x_\lambda,d/2)}m_{\lambda}\p^{\alpha}Q(y-z)dzdy,
\enq
which implies that 
\beq
\sum_{l\neq\lambda}\fint_{B(x_l,r)}m_{\lambda}\p^\alpha Q(y-x_\lambda)dy&=&\sum_{l\neq\lambda}\fint_{B(x_l,a)}\fint_{B(x_\lambda,d/2)}m_{\lambda}\p^{\alpha}Q(y-z)dzdy\\
&=&\frac{9}{2\pi^2a^3d^3}\int_{B(x_l,a)}\int_{\RR^3\setminus B(x_l,d/2)}\Phi(z)\p^\alpha Q(y-z)dzdy.
\enq
We also notice that  
\beq
\big|\Phi(z)\p^\alpha Q(y-z)\big|\leq C|\Phi(z)|\frac{1}{|y-z|^5}.
\enq
Therefore we obtain that 
\beq
\big|W^{2}_{l,\alpha}\big|\leq Cd^{-3}\int_0^a rdr \int_{B(x_l,a)} \int_{\RR^3\setminus B(x_l,d/2)}|\Phi(z)| \dfrac{1}{|y-z|^5}dzdy.
\enq
By a similar argument as before, we obtain that 
\beq
\big|W^2_{l,\alpha}\big|&\leq& C a^2d^{-3}\int_{B(x_l,a)}\int_{\RR^3\setminus B(y,d/2-a)}|\Phi(z)|\frac{1}{|y-z|^5}dzdy\\
&\leq& a^{5-3/q}d^{-3}\big\|\int_{\RR^3\setminus B(0,d/2-a)}|\Phi(\cdot-z)|\frac{1}{|z|^5}dz\big\|_{L^q(B(x_l,a))},
\enq
which implies that, since $a^3/d^3  < \varepsilon_0 <<1$:
\beqq\label{W2l}
\begin{split}
\Big(\sum_{l=1}\big|W^2_{l,\alpha}\big|^q\Big)^{1/q}&\leq C a^{5-3/q}d^{-3}\big\|\int_{\RR^3\setminus B(0,d/2-a)}|\Phi(\cdot-z)|\frac{1}{|z|^5}dz\big\|_{L^q(\RR^3)}\\
&\leq C(q,\varepsilon_0)\big(\frac{a}{d}\big)^{5-\frac{3}{q}}\Big(\sum_{l=1}^N|m_l|^q\Big)^{1/q}.
\end{split}
\enqq

\textbf{Step 3.} At last we deal with $W^3_{l,\alpha}$. We recall that 
\beq
W^{3}_{l,\alpha}=\dfrac{a^3}{3}\fint_{B(x_l,a)}\big(\sum_{l\neq\lambda}\int_0^{\frac{d}{2}}\left(\frac{8r^4}{d^3} - r\right)\fint_{B(x_\lambda,r)}m_{\lambda}\p^\alpha Q(y-z)dzdr\big)dy.
\enq
Again by the fact that $\p^{\alpha}Q$ is harmonic in $\RR^3\setminus\{0\}$ for any $|\alpha|=2$, we obtain that for $y\in B(x_l,a),$  $l\neq\lambda$
 and $r<d/2$:
\beq
\fint_{B(x_\lambda,r)}m_{\lambda}\p^\alpha Q(y-z)dz=\fint_{B(x_\lambda,d/2)}m_{\lambda}\p^\alpha Q(y-z)dz.
\enq
By a similar argument as in step 2, we obtain that 
\beq
\big|W^3_{l,\alpha}\big|&\leq& C d^{-3}\int_{B(x_l,a)}\int_0^{d/2}rdr\sum_{l\neq\lambda}\int_{B(x_\lambda,d/2)}|\Phi(z)|\frac{1}{|y-z|^5}dzdy\\
&\leq& Cd^{-1}\int_{B(x_l,a)}\sum_{l\neq\lambda} \int_{B(x_\lambda,d/2)}|\Phi(z)|\frac{1}{|y-z|^5}dzdy\\
&\leq&C d^{-1}\int_{B(x_l,a)}\int_{\RR^3\setminus B(y,d/2-a)}|\Phi(z)|\frac{1}{|y-z|^5}dzdy\\
&\leq& C a^{3-3/q}d^{-1}\big\|\int_{\RR^3\setminus B(0,d/2-a)}|\Phi(\cdot-z)|\frac{1}{|z|^5}dz\big\|_{L^q(B(x_l,a))},
\enq
and we conclude that:
\beqq\label{W3l}
\begin{split}
\Big(\sum_{l=1}\big|W^3_{l,\alpha}\big|^q\Big)^{1/q}&\leq C a^{3-3/q}d^{-1}\big\|\int_{\RR^3\setminus B(0,d/2)}|\Phi(\cdot-z)|\frac{1}{|z|^5}dz\big\|_{L^q(\RR^3)}\\
&\leq C(q)\big(\frac{a}{d}\big)^{3-\frac{3}{q}}\Big(\sum_{l=1}^N|m_l|^q\Big)^{1/q}.
\end{split}
\enqq
By combining  \reff{W1l}, \reff{W2l} and \reff{W3l}, we obtain the
expected result since $a/d < \varepsilon_0 <<1$:
\beq 
\big(\sum_{l=1}^N\Big|W_{l,\alpha}\Big|^q\big)^{1/q}&\leq& \big(\sum_{l=1}^N\Big|W^{1}_{l,\alpha}\Big|^q\big)^{1/q}+\big(\sum_{l=1}^N\Big|W^{2}_{l,\alpha}\Big|^q\big)^{1/q}+\big(\sum_{l=1}^N\Big|W^{3}_{l,\alpha}\Big|^q\big)^{1/q}\\
&\leq& C(q,\varepsilon_0)\big(\frac{a}{d}\big)^{3-\frac{3}{q}}\big(\sum_{l=1}^N|m_{l}|^q\big)^{1/q}.
\enq
The first statement of the lemma is proved.

\medskip

\textbf{Part 2.} In this part we give a proof for the second item. By definition of $F_l$ and Lemma \ref{mean-value}, we have that for any $x\in B(x_l,2a)$
\beq
F_l(x)&=&\sum_{l\neq\lambda} m_\lambda\fint_{B(x_\lambda,d/2)}\nabla V(x-y)dy+ \dfrac{1}{3} \sum_{l\neq\lambda}m_\lambda\int_0^{\frac{d}{2}}\left(\frac{8r^4}{d^3} - r\right)\fint_{B(x_\lambda,r)}\nabla Q(x-y)dydr.
\enq
According to the fact that $Q$ is harmonic outside the origin, the second term on the right side of the above equation can be written as
\begin {multline*}
\sum_{l\neq\lambda}m_\lambda\int_0^{\frac{d}{2}}\left(\frac{8r^4}{d^3}-r\right)\fint_{B(x_\lambda,r)}\nabla Q(x-y)dydr \\
\begin{aligned}
&=\sum_{l\neq\lambda}m_\lambda\int_0^{\frac{d}{2}}\left(\frac{8r^4}{d^3}-r\right)\fint_{B(x_\lambda,d/2)}\nabla Q(x-y)dydr\\
&=\frac{33}{40\pi d}\sum_{l\neq\lambda} m_{\lambda}\int_{B(x_l,d/2)}\nabla Q(x-y)dy,
\end{aligned}
\end{multline*}
for any $x\in B(x_l,2a)$. Therefore for any $x\in B(x_l,2a)$, we have:
\beq
F_l(x)=\frac{6}{\pi d^3}\int_{\RR^3\setminus B(x_l,d/2)}\Phi(y)\nabla V(x-y)dy+\frac{11}{40 \pi d}\int_{\RR^3\setminus B(x_l,d/2)}\Phi(y)\nabla Q(x-y)dy,
\enq
where $\Phi$ is defined in \reff{Phi}.

\medskip

Now we start to prove \reff{Fl-H^1}. By the above argument and since the $B(x_l,2a)$ are disjoint, we have
\begin{multline}\label{F_l}
\sum_{l=1}^N\|\nabla F_l\|^2_{L^2(B(x_l,2a))}\leq C \dfrac{1}{d^{6}}\sum_{|\alpha|=2}\sum_{l=1}^N\|\int_{\RR^3\setminus B(x_l,d/2)}\Phi(y)\partial^\alpha V(\cdot-y)dy\|_{L^2(B(x_l,2a))}^2\\
+\dfrac{1}{d^{2}}\sum_{|\alpha|=2}\sum_{l=1}^N\|\int_{\RR^3\setminus B(x_l,d/2)}\Phi(y)\partial^\alpha Q(\cdot-y)dy\|_{L^2(B(x_l,2a))}^2.
\end{multline}
In order to control the right-hand side of the above inequality, we first notice that for each $l=1,\dots,N,$ any $|\alpha|=2$ and $x \in B(x_l,2a):$
\beq
\int_{\RR^3\setminus B(x_l,d/2)}\Phi(y)\partial^\alpha V(x-y)dy&=&\int_{\RR^3\setminus B(0,d/2)}\Phi(x-y)\partial^\alpha V(y)dy\\
&&+\Big(\int_{\RR^3\setminus B(x_l,d/2)}-\int_{\RR^3\setminus B(x,d/2)}\Big)\Phi(y)\partial^\alpha V(x-y)dy.
\enq

By  similar arguments as in \textbf{Part 1} of the proof, we obtain that 
\beq
\begin{split}
\sum_{|\alpha|=2}\sum_{l=1}^N\|\int_{\RR^3\setminus B(0,d/2)}\Phi(\cdot-y)\partial^\alpha V(y)dy\|_{L^2(B(x_l,2a))}^2\leq Cd^3\sum_{l=1}^N|m_l|^2
\end{split}
\enq
and
\beq
\begin{split}
\sum_{|\alpha|=2}\sum_{l=1}^N&\|\Big(\int_{\RR^3\setminus B(x_l,d/2)}-\int_{\RR^3\setminus B(x,d/2)}\Big)\Phi(y)\partial^\alpha V(x-y)dy\|^2_{L^2(B(x_l,2a))}\\
\leq&C\left| \ln\Bigl(\frac{d+2a}{d-2a}\Bigr)\right|^{2}d^{3}\sum_{l=1}^N|m_l|^2
\end{split}
\enq
Therefore we have
\begin{multline}\label{F_l-N}
\sum_{|\alpha|=2}\sum_{l=1}^N\|\int_{\RR^3\setminus B(x_l,d/2)}\Phi(y)\partial^\alpha V(\cdot-y)dy\|_{L^2(B(x_l,2a))}^2 \\
\leq Cd^3\left(1+\left|\ln\Bigl(\frac{d+2a}{d-2a}\right|^2\Bigr)\right)\sum_{l=1}^N|m_l|^2. \phantom{12345678}
\end{multline}
Now we turn to deal with the second term of the right side of \reff{F_l}. We first notice that for any $|\alpha|=2,$ $l=1,\ldots,N$ and $x \in B(x_l,2a)$,
there holds:
\beq
\Big|\int_{\RR^3\setminus B(x_l,d/2)}\Phi(y)\partial^\alpha Q(x-y)dy\Big|\leq C\int_{\RR^3}|\Phi(y)|\frac{\mathbf{1}_{|x-y|>d/2-2a}}{|x-y|^5}dy,
\enq
 Since $a^3/d^3 < \varepsilon_0 <<1$  we obtain via a standard Young inequality for convolutions that 
\beqq\label{F_l-E}
\sum_{|\alpha|=2}\sum_{l=1}^N\|\int_{\RR^3\setminus B(x_l,d/2)}\Phi(y)\partial^\alpha Q(\cdot-y)dy\|_{L^2(B(x_l,2a))}^2\leq C(\varepsilon_0)d^{-1}\sum_{l=1}^N|m_l|^2.
\enqq

Finally, combining \reff{F_l}, \reff{F_l-N} and \reff{F_l-E} and remarking that $d\leq diam(K)$, we have that 
\beq
\sum_{l=1}^N\|\nabla F_l\|^2_{L^2(B(x_l,2a))}\leq C(\varepsilon_0)d^{-3}\sum_{l=1}^N|m_l|^2.
\enq
This ends the proof of the second item and the proof of the lemma.
\end{proof}

\end{document}